\documentclass[envcountsame,orivec,runningheads]{llncs}
\usepackage{a4,graphics,amsmath,amssymb,amsfonts,xspace,stmaryrd,textgreek,xcolor}
\usepackage{hyperref,breakurl}             
\usepackage[normalem]{ulem}
\usepackage[capitalise]{cleveref}
\usepackage{tikz,pgfplots,csvsimple}
\usepackage{stix}
\usepackage{accents,ushort}
\usepackage{makecell}
\usepackage{algpseudocode}
\algblockdefx{Prog}{EndProg}[1]{\textbf{input} #1}[1]{\textbf{return} #1}
\newcommand{\myprime}{\textquotesingle\xspace}
\newcommand{\myprimes}{\textquotesingle\textquotesingle\xspace}
\newcommand{\IR}{\mathbb{R}}

\newcommand{\IF}{\mathbb{F}}
\newcommand{\IC}{\mathbb{C}}
\newcommand{\IZ}{\mathbb{Z}}
\newcommand{\IN}{\mathbb{N}}

\newcommand{\IK}{\mathbb{K}}

\newcommand{\IT}{\mathbb{T}}
\newcommand\omicron{\ensuremath{o}}
\newcommand{\error}{\epsilon}

\newcommand{\uobarZ}{\underaccent{\bar}{\bar\IZ}}

\newcommand{\poly}{\operatorname{poly}}
\newcommand{\polylog}{\operatorname{polylog}}

\newcommand{\Card}{\operatorname{Card}}
\newcommand{\bin}{\mathrm{bin}}

\newcommand{\Round}{\operatorname{\sf Round}}
\newcommand{\sign}{\operatorname{sign}}
\newcommand{\dom}{\operatorname{dom}}
\newcommand{\cost}{\operatorname{cost}}
\newcommand{\mychoose}{\texttt{choose}}
\newcommand{\id}{\operatorname{id}}
\newcommand{\concat}{{}^\frown}

\newcommand{\sdzero}{\textup{\texttt{0}}\xspace}
\newcommand{\sdone}{\textup{\texttt{1}}\xspace}
\newcommand{\TWO}{\{\sdzero,\sdone\}}
\newcommand{\ONE}{\{\sdone\}}
\newcommand{\calA}{\mathcal{A}}
\newcommand{\calB}{\mathcal{B}}

\newcommand{\calF}{\mathcal{F}}

\newcommand{\calO}{\mathcal{O}}

\newcommand{\calL}{\mathcal{L}}

\newcommand{\calP}{\mathcal{P}}
\newcommand{\calS}{\mathcal{S}}
\newcommand{\calT}{\mathcal{T}}
\newcommand{\frakF}{\mathfrak{F}}
\newcommand{\Machine}{\mathcal{M}}

\newcommand{\mapstoto}{\Mapsto}

\newcommand{\ERC}{\textsf{\rm ERC}\xspace}
\newcommand{\true}{\texttt{t}\xspace}
\newcommand{\false}{\texttt{f}\xspace}
\newcommand{\unknown}{\texttt{u}\xspace}

\newcommand{\eg}{e.\,g.\xspace}
\newcommand{\ie}{i.\,e.\xspace}

\newcommand{\myn}{\mathbf{n}}
\newcommand{\naive}{na\"{\i}ve\xspace}
\newcommand{\naively}{na\"{\i}vely\xspace}
\newcommand{\Naive}{Na\"{\i}ve\xspace}

\newcommand{\mydoi}[1]{\href{http://doi.org/#1}{\texttt{doi:#1}}}
\makeatletter
\DeclareRobustCommand*{\bfseries}{%
  \not@math@alphabet\bfseries\mathbf
  \fontseries\bfdefault\selectfont
  \boldmath
}
\makeatother

\newcommand{\COMMENTED}[1]{}

\title{Algorithmic Cost in \emph{Exact Real Computation}}
\author{Jihoon Hyun\inst{1} \and Holger Thies\inst{2} \and Martin Ziegler\inst{1}}  
\authorrunning{Jihoon Hyun \and Holger Thies \and Martin Ziegler}
\institute{KAIST, School of Computing, Daejeon, Republic of Korea 
\and Kyoto University, Graduate School of Human and Environmental Studies, Japan}
\date{}

\begin{document}
\setcounter{page}{0}
\maketitle

\begin{abstract}
Turing completeness of a programming language or system characterizes its expressive power;
and the strong Church-Turing hypo-/thesis refines such from qualitative to polynomial-time equivalence.
\emph{Exact Real Computation} (\ERC) is a novel numerical programming language paradigm: 
for the imperative processing of continuous data as entities 
appearing as exact, \ie devoid of rounding errors \linebreak[4] [\mydoi{10.1007/978-3-662-44199-2\_107}].
\ERC has been designed \linebreak[4] [\mydoi{10.46298/lmcs-20(2:17)2024}] 
as convenient and practical alternative, namely
proven qualitatively equivalent, to the 
Turing machines originally underlying Computable Analysis
[\mydoi{10.1007/978-3-642-56999-9},
\mydoi{10.1007/978-1-4684-6802-1}].

The present work quantitatively strengthens this qualitative Turing-completeness:
We assign bit-costs to \ERC's operational primitives 
(including partial/multivalued tests) in such a way that
any real function incurring polynomial cost becomes
Turing-computable in polynomial time, and vice versa.
Runtime measurements on implementations in the
\textsf{iRRAM} \texttt{C++} library 
confirm our theoretical performance predictions.
\end{abstract}

\setcounter{tocdepth}{2}
\setcounter{secnumdepth}{2}
\renewcommand{\contentsname}{}
\section*{Table of Contents}
\begin{center}\small
\begin{minipage}[c]{0.98\textwidth}\vspace*{-8ex}%
\tableofcontents
\end{minipage}
\end{center}

\newpage
\section{Introduction}
\label{s:Introduction}
90 years after its introduction, the Turing Machine is still the yardstick
for both computability and computational complexity.
Within Chomsky's Hierarchy for example, 
Type-0 grammars are characterized as Turing-complete; 
and the Myhill-Landweber-Kuroda Theorem characterizes Type-1 as \textsf{NSPACE}$(n)$.
Actually programming a Turing Machine is cumbersome, though.
Practically convenient high-level programming languages hide awkward hardware details:
\Naive arithmetic on, say, \texttt{bytes} can cause wrap-around and code malfunction \cite{NuclearGandhi};
whereas \textsf{Python}'s \texttt{int} and \textsf{Java}'s \texttt{BigInteger} 
and \textsf{GMP}'s \texttt{mpz\_t}
relieve the programmer from worrying about the magnitude of integers---by
dynamically tracing and tacitly adapting to the data's size during execution, 
such as to maintain agreement with mathematical $\IZ$.
For performance analysis and prediction, bit-lengths $n$ and their evolution
throughout computations do remain crucial, though.

Indeed, 
the \emph{Integer Register Machine} model captures the ultimate capabilities of Turing machines \cite{Davis1965}, quantitatively.
Repeated multiplications can blow up results;
quantitatively refined runtime estimates therefore
assign bit (as opposed to unit) costs\footnote{%
Here and in the sequel we may concisely omit constant factors $\calO()$ from cost analyses.} 
to each primitive---depending 
on the binary length $n=1+\lfloor\log_2|N|\rfloor$ of the integer/s $N$ processed \cite{CookReckhow73}; 
and bound how each operation affects $n$:
see Figure~\ref{f:IntCost}. 
This model \emph{with} bit-costs 
underlies most discrete algorithm analyses, 
\eg in Computer \emph{Algebra} \cite{MCA}.

\begin{figure}[htb]\centerline{%
\begin{tabular}{c@{\quad}|@{\quad}c@{\quad}c@{\quad}c@{\quad}c@{\quad}c@{\quad}c@{\quad}c@{\quad}c}
\textbf{operation} & \texttt{>}  &  \texttt{+}/\texttt{-}  &  \texttt{\texttimes} & \texttt{\textdiv} & 
\texttt{isqrt} & \texttt{isprime} \\ 
\hline 
\textbf{cost} 	&  $n$ &   $n$     &   $n\log n$   & $n\log n$ & 
$n$ & poly$(n)$ \\
\textbf{result $n'$} &  1   &   $n+1$   &    $2n$ &   $n$  &  
$\lceil n/2\rceil$ & 1
\end{tabular}}
\caption{\label{f:IntCost}Bit-costs of common integer operations and result worst-case length $n'$}
\end{figure}

\begin{remark}
\label{r:IntCost}
Comparison is a predicate, returning one bit.
Recall the break-through \cite{IntegerMultiplication} 
performing $\times$ in asymptotic bit-cost $M(n)=n\log n$;
and see for instance \cite[Theorem~9.8]{MCA} regarding $\div$, 
and \cite{KaratsubaSqrt} for \texttt{isqrt}=$\lfloor\sqrt{\cdot}\rfloor$,
and \cite{Dietzfelbinger} for predicate \texttt{isprime}. 
\end{remark}
We similarly assign bit-costs to the \emph{real} register machine model \ERC,
whose primitives process real numbers (seemingly) exactly---including
partial and fuzzy sign tests as well as fast Cauchy limits. 
As a decidedly numerical formal programming language,
these local costs additionally depend on the current precision parameter $m$, 
which \ERC hides and handles on behalf of the user programmer.
In addition to the integer Register Machine and to $n$ evolving forward,
from inputs through program execution, said $m$ grows `backwards': 
starting with the desired output accuracy $p$, and ensured large enough throughout 
such that intermediate real signs can be determined reliably.
Unifying common practical approaches for finding such $m$ from $p$ for a given algorithm $\calA$ and input $\bar x$,
we capture its evolution as reverse composition of the moduli of continuity
of the individual primitive operations and predicates executed by a
collection of computation trees which may invoke each other and possibly recursively so.
Our Main Theorems~\ref{t:Simul1}+\ref{t:Simul2} then assert that any real function computable
in polynomial time on a Turing Machine can be expressed in \ERC within polynomial cost, and vice versa.
Implementations and empirical evidence suggest that both might in fact differ at most quadratically.

\begin{definition}[Abstract Cost Measure]
\label{d:Cost}
Fix a programming language $\calL$, \ie with formal syntax and semantics.
A (global) \emph{cost measure} on $\calL$
assigns some non-negative real number $\cost(\calA,\bar x)$, or infinity,
to any---not necessarily terminating---execution of a program $\calA$ over $\calL$ on input $\bar x$ .

A \emph{local} cost measure assigns some non-negative real number $\cost(\omicron,\bar y)$, or infinity,
to any execution of a primitive $\omicron$ of $\calL$ with argument $\bar y$.
Its \emph{induced} (global) cost measure assigns to any execution of program $\calA$ on input $\bar x$ 
the sum $\sum_j\cost(\omicron_j,y_j)$ of the local costs 
of the primitives $\omicron_j$ performed with arguments $\bar y_j$.
\end{definition}

\subsection{Overview}
\label{ss:Overview}

Section~\ref{s:Real} recalls and collects selected aspects of real number computation:
with user-controlled and tracked accuracy (Subsection~\ref{ss:FPCost})
and according bit-cost assignments and example forward analyses (Subsection~\ref{ss:Examples});
modulus of continuity capturing stability and error propagation (Subsection~\ref{ss:Stability});
and a brief recap of the \emph{Exact Real Computation} paradigm \ERC (Subsection~\ref{ss:ERC}),
which delegates accuracy handling to software \cite{MullerZ14} 
and is elaborated in Sections~\ref{s:Contribution} and \ref{s:ERC2}.

Section~\ref{s:Contribution} contains the first half of our contribution:
Definition~\ref{d:Contribution} assigns local bit-costs to the basic operations comprising \ERC.
These costs naturally depend on both the arguments' length and the working precision parameters $n$ and $m$.
The evolution of $n$ is immediate, similarly to the integer register machine;
starting with the input length, its growth through intermediate values 
can be bounded by forward induction on the operations performed.
Determining a sufficient working precision $m$ on the other hand 
is particular to numerical algorithms
and discussed in Subsection~\ref{ss:Stability2}. 
In \ERC, said precision evolves backwards:
from desired output precision $p$ via intermediate requirements 
for reliable real tests, towards the real inputs and constants
as mentioned in Subsection~\ref{ss:ERC}b) and Problem~\ref{p:Stability2}.
In order to formally capture this dependency, unifying
\emph{bottom-up} and \emph{top-down} and mixed methods from practice \cite[\S2.2]{realLib},
Subsection~\ref{ss:Modulus2} extends the quantitative stability notion from Subsection~\ref{ss:Stability}
to multivariate primitive operations;
and Subsection~\ref{ss:Choose} further on to multivalued predicates (such as fuzzy tests) arising in \ERC.
Subsection~\ref{ss:Round} demonstrates these concepts by analyzing two example \ERC programs for multivalued integer rounding:
one incurring bit-cost proportional to the value (=unary length) of the real argument,
and the other more sophisticated one proportional to its logarithm (=binary length).
The latter is the key to Subsection~\ref{ss:Simul1} establishing one direction of our contribution about quantitative equivalence:
any real function computed by a polynomial-time oracle Turing machine in the sense of \cite[\S2.4+\S2.5]{Ko91}
can also be computed by an \ERC program incurring polynomial cost.

For the converse, culminating in Subsection~\ref{ss:Simul2},
Section~\ref{s:ERC2} assigns moduli of continuity by backward induction/composition to \ERC programs:
the latter first formalized as Straight-Line Programs without branching (Subsection~\ref{ss:SLP}),
then generalized to Computation Trees with multivalued tests (Subsection~\ref{ss:CT}),
and finally (Subsection~\ref{ss:CT}) as Computation Forests \ie
collections of Computation Trees which may invoke (``call'') each other, possibly recursively.

Implementations and empirical evaluations in the \textsf{iRRAM} \texttt{C++} library 
in Appendix~\ref{a:Experiments} 
suggest that the thus assigned and predicted abstract costs actually
agree quadratically (rather than merely polynomially) with practical runtimes.
This includes an exponentially `steep' function (Remark~\ref{r:Modulus}d),
whose realization and cost in \ERC calling the unbounded real exponential and logarithm function
are elaborated in Appendix~\ref{a:Steep}.

\medskip
Our previous conference version \cite{MCU2026} essentially comprises
Sections~\ref{s:Real} and \ref{s:Contribution} of the present work
including Theorem~\ref{t:Simul1}, but not Section~\ref{s:ERC2}
with the converse Theorem~\ref{t:Simul2} nor the Appendices.

\subsection{Related Work}
\label{ss:Related}
Programming Language Theory is usually concerned with the syntax and semantics of some high-level programming language $\calL$.
Such $\calL$ may have sub-recursive \cite{Automata} or super-recursive power \cite{Hypercomputation}.
Turing-completeness (aka computational universality) captures capabilities equal to that of a Turing machine,
asserting both realizability and expressiveness---but disregarding efficiency \cite[\S14]{Minsky67}.
The integer register machine is equivalent to the Turing machine model on this qualitative computability level
but, depending on the instruction set, not necessarily with respect to polynomial-time complexity---not
to mention even finer the degree of polynomial time \cite[\S2.2.2]{Boas90}.
This can be mended by proceeding from unit (=command counting) cost to suitably designed local bit-costs,
as in Figure~\ref{f:IntCost}. 

Programming Language Theory is usually concerned with the syntax and semantics of some high-level programming language $\calL$.
Such $\calL$ may have sub-recursive \cite{Automata} or super-recursive power \cite{Hypercomputation}.
Turing-completeness (aka computational universality) captures capabilities equal to that of a Turing machine,
asserting both realizability and expressiveness---but disregarding efficiency \cite[\S14]{Minsky67}.
The integer register machine is equivalent to the Turing machine model on this qualitative computability level
but, depending on the instruction set, not necessarily with respect to polynomial-time complexity.
This can be mended by proceeding from unit (=command counting) cost to suitably designed local bit-costs,
as in Figure~\ref{f:IntCost}. 

\begin{remark}
\label{r:ChurchTuring}
Various versions of the Church-Turing Hypo-/Thesis \cite{Deutsch85,Yao03,Ziegler09,BeggsTucker14} 
connect among the following conditions, 
primarily for functions $f$ on natural numbers:
\begin{itemize}
\item
$f$ is computable ~(i) by a Turing machine, ~(ii) by a register machine, ~(iii) on a \emph{physical} computing device.
\item
$f$ is computable in \emph{polynomial} time ~(i\myprime) by a Turing machine, ~(ii\myprime) by a register machine w.r.t. bit-cost, ~(iii\myprime) on a physical computer.
\item
Given $d\in\IN$, 
$f$ is computable in time $\calO(n^d\cdot\polylog n)$ by ~(i\myprimes) a multitape Turing machine or ~(ii\myprimes) by a register machine w.r.t. bit-cost.
\end{itemize}
Additionally, for a programming language $\calL$ equipped with a
cost measure in the sense of Definition~\ref{d:Cost}, 
consider the following properties:
\begin{itemize}
\item
$f$ can be ~(iv) expressed in $\calL$, ~(iv\myprime) within polynomial cost, ~(iv\myprimes) of same polynomial degree up to polylogarithmic factors.
\qed\end{itemize}
\end{remark}
\noindent
Abstract cost measures were introduced in \cite{Blum67}.
We focus here on the imperative paradigm:
(Pure) functional/declarative programming, while elegant and nowadays also usually fast in practice, 
reduces or removes user control over the flow of execution---which 
may undermine rigorous asymptotic performance guarantees (iv\myprime/iv\myprimes).
Modern languages like \textsf{Python} or \textsf{Java} offer some blend of both,
imperative and functional; compare Subsection~\ref{ss:ERC}c) below.
See \cite{LambdaCost,LambdaMachine} about assigning local costs to 
Lambda Calculus in order to strengthen its Turing-completeness (iv) to polynomial equivalence (iv\myprime).
The latter quest is different from designing a programming system in the first place:
in order to have expressive power equivalent to polynomial-time Turing computation \cite{Lago22}
or some other form of built-in resource bound as in Linear Logic.

\begin{example}
\label{x:C}
\begin{enumerate}
\item[a)]
For the `core' (\ie including pointers, but no string functions)
\textsf{C} programming language with unit costs,
(iv\myprimes) is strongly polynomially equivalent to (i\myprimes) and (ii\myprimes):
because its basic data types (like \texttt{char}, \texttt{short}, \texttt{int}, \texttt{*}) 
have constant length and operational costs---enabling 
reliable performance predictions, such as for embedded systems.
\item[b)]
The polynomial equivalence from (a) does extend to string functions, when charging them
cost proportional to the length $n$ of the string being operated on.
Note that \textsf{string.h} relieves the programmer from keeping track of the lengths $n$ of the strings being processed;
while aforementioned lengths do remain crucial for performance analysis and prediction (iv\myprime/iv\myprimes).
\item[c)]
Like \textsf{string.h}, the \textsf{GMP} library takes care of the lengths $n$ of integers,
while they do remain crucial in performance analyses (iv\myprime/iv\myprimes).
For \textsf{C} with \textsf{GMP}, 
(iv\myprimes) is still equivalent to (i\myprimes) and (ii\myprimes)
when assigning bit-costs according to Figure~\ref{f:IntCost}.
\item[d)]
Any programming system $\calL$ with recursively enumerable semantics 
can trivially be assigned equivalent (iv\myprimes) abstract global cost:
by defining $\cost(\calA,\bar x)$ as the number of steps the Turing machine simulating $\calA$ makes on input $\bar x$.
These global costs may even be induced from `artificially' local ones, 
namely depending on the entire current configuration $\vec y$ of the simulating Turing machine.
\qed\end{enumerate}\end{example}
`Reasonable' local cost assignments should depend only on the current argument $\bar y$ to the current programming primitive $\omicron$.
In fact it is customary (as in Figure~\ref{f:IntCost})
to express local costs in dependence only on the length $n$ of the arguments $\bar y$ to $\omicron$,
rather than on $\bar y$ itself. 
In this case, qualitative polynomial-time equivalence of (ii\myprime) or (iv\myprime) to (i\myprime)
means to bound the largest variable contents and the number of operations independently.
To \ERC we shall assign local costs depending on the length $n$ and on the working precision $m$.

\medskip\noindent
Regarding \emph{real} (as opposed to integer) functions $f$, 
the theoretical literature is mostly divided between two major models of computation:
the algebraic Blum-Shub-Smale (BSS) machine aka \emph{real-RAM} \cite{ACT97,BCSS98}
and the `analytic' framework of \emph{Computable Analysis} based on Turing Machines \cite{Ko91,Wei00}.
Both give rise to their own elaborate but incomparable \cite{Emperor} computability 
and complexity classifications. 

Recent works \cite{BH98,MullerZ14,Sewon23,Clerical,PBC+24b,Sewon25} reconcile the two models,
combining the best of both worlds in theory and in practice \cite{Mue01},
under the moniker \emph{Exact Real Computation} (\ERC); 
see Subsection~\ref{ss:ERC}.
Indeed, modifying the \naive semantics of real tests 
and adding a suitable Cauchy limit functional (Example~\ref{x:Limit})
yields a formal programming language (iv) 
qualitatively equivalent \cite[\S5.3]{PBC+24b} to Computable Analysis (i).
And our contribution refines this qualitative to polynomial equivalence (i\myprime)=(iv\myprime):
by carefully assigning bit-costs to real operations,
now depending additionally on the working precision parameter $m$.
Beyond handling (and hiding) the length $n$ of numbers like the integer register machine, 
\ERC also takes care of the working precision $m$ on behalf of the user programmer;
while naturally both still do enter in performance analyses.

In similar spirit, the class of real functions
qualitatively computable in Computable Analysis (i) and those within polynomial time (i\myprime)
both have recently been characterized 
(iii+iii\myprime) in terms of Shannon's \emph{General Purpose Analog Computer} GPAC
\cite{BournezGracaPouly17,BlancBournez22}.

\section{Real Number Computation}
\label{s:Real} 

Both Theory and Practice of real number computation are rich topics with extensive literature.

Numerical practice usually builds on (\eg IEEE\,754) \texttt{double} precision floating-point numbers 
with hardware support.
Operations on such data items take time bounded independently of its argument/s,
which is theoretically captured by the aforementioned \emph{real-RAM} model \cite{ComputationalGeometry,BCSS98}:
having registers for reading, storing, processing, and outputting real numbers---as 
entities at unit cost exactly, without incurring nor propagating rounding errors.
This abstraction induces a structural computability and complexity theory parallel to the discrete setting 
\cite{FournierKoiran98,MeerZiegler08,ZieglerKoolen08,SchaeferCardinalMiltzow24}.
Precision questions have to be addressed separately \cite{Condition};
see also Subsection~\ref{ss:FPCost} below.

Computable Analysis on the other hand revolves arguably around precision primarily:
formalized in the Type-2 Turing Machine processing converging sequences of approximations \cite{Wei00},
or the Oracle Turing Machine \cite{Ko91} also translating among families of approximations of variable (but guaranteed) precision
\cite{Grz57}. Both Type-2 and Oracle Turing Machines are 
important for theoretical and historical reasons, but practically inconvenient:
thus the \ERC paradigm recalled in Subsection~\ref{ss:ERC}.
Indeed, both machines are (i\myprime) polynomial-time but not (i\myprimes) linear-time equivalent \cite[Proposition~1.3]{LimZiegler25}, 
essentially because simulating one random-access oracle access on a sequential tape incurs linear overhead.

\begin{remark}
\label{r:MainThm}
In Computable Analysis,
any computable function must necessarily be continuous \cite[\S4.3]{Wei00}; 
and polynomial time requires polynomial modulus of continuity,
see Remark~\ref{r:Modulus}e) below.
\end{remark}

\subsection{User-Controlled Accuracy}
\label{ss:FPCost}

Processing continuous data traditionally proceeds by discretization: up to some suitable error $\varepsilon>0$,
and then operating on these approximations---leaving it to the developer to keep track of error propagation.
Formally consider the set $\IF$ of `parameterized' \emph{fixed}-point numbers
$x=(\pm \: b_{n-1}  \cdots b_1 \: b_0 \: \texttt{.} \: b_{-1} \cdots b_{-m})_2$
with at most $n\in\IZ$ binary digits in front of the radix point
(\ie of absolute value $\leq2^n$) and with at least $m\in\IZ$ reliable bits (=precision) after the radix point.
Basic operations on such data are easily seen to incur computational bit-costs,
and to affect parameters $(n,m)$, as collected in Figure~\ref{f:FPCost}; 
see Subsection~\ref{ss:Stability}.

\begin{figure}[htb]\centerline{%
\begin{tabular}{c@{\quad}|@{\quad}c@{\quad}c@{\quad}c@{\quad}c@{\quad}c@{\quad}c@{\quad}c}
\textbf{operation} & $\sign$   &   $\pm$  &  $\times$       & $1/\cdot$ & $\sqrt{\cdot}$ & $\pi$ \\ 
\hline
\textbf{cost} 	&  \makecell{$\scriptstyle {n+m+}$ \\[-0.5ex] $\scriptstyle {\log|n|}$} &  \makecell{$\scriptstyle {n+m+}$ \\[-0.5ex] $\scriptstyle {\log|n|}$}   & 
\makecell{$\scriptstyle {(n+m)\cdot}$ \\[-0.5ex] $\scriptstyle {\log(n+m)}$} & 
\makecell{$\scriptstyle {(n+m)\cdot}$ \\[-0.5ex] $\scriptstyle {\log(n+m)}$} &
\makecell{$\scriptstyle (n+m)\cdot$ \\[-0.5ex] $\scriptstyle {\log(n+m)}$} &
\makecell{$\scriptstyle {m'\cdot}$ \\[-0.5ex] $\scriptstyle {\log^2m'}$} \\[1.5ex]
\textbf{result $n'$} &  0  &  $n+1$   &     $2n$        &   $n'$    & $\lceil n/2\rceil$  & 2 \\ 
\textbf{result $m'$} &  $\infty$  &  $m-1$   & $m-n-2$      &  ${m-2n'-2}$ &  ${m+n'-1}$
& $m'$  
\end{tabular}}
\caption{\label{f:FPCost}Bit-costs of operations on variable-precision numbers and (worst-case) result parameters}
\end{figure}
\begin{remark}
\label{r:FPCost}
Our asymptotic analyses here and in the sequel count only entire reliable bits,
while practice knows the benefits of `fractional' accuracy propagation \cite{Lester01};
see also Remark~\ref{r:OverEstimate} below.
\begin{enumerate}
\item[a)]
For similar reasons, $n$ here usually denotes only an upper bound 
on the number of binary digits in front of the radix point;
and $m$ denotes a lower bound on the number of reliable bits
after the radix point.
\item[b)]
Call $\bar b$ \emph{normalized} if $b_{n-1}\neq0$.
Note that $x\neq0$ in such normalized representation 
has least $n=\myn(x):=1+\lfloor\log_2|x|\rfloor$. 
Zero cannot be normalized, $\myn(0):=-\infty$.
\item[c)]
We deliberately allow $n,m<0$ as long as $m>-n$:
corresponding to \emph{floating}-point numbers 
with mantissa $\bar b=(\pm b_{n-1} \texttt{.} b_{n-2} \cdots b_{-m})$ of length $n+m$ 
to binary exponent $n$. 
The overall bit-size of a thus represented number $x$
is, up to some constant for delimiters, $n+m+\log|n|>0$:
\item[d)]
Hence the bit-cost charged for $\sign(x)$ in Figure~\ref{f:FPCost}.
Again, in case $x=0$, the cost becomes infinite---in 
agreement with real equality well-known equivalent 
to (the complement of) the Halting Problem \cite{Rice54}; 
see also Subsection~\ref{ss:ERC}a+d) below.
\item[e)]
Floating-point addition (and subtraction) with relative errors is well-known prone to cancellation.
In our fixed-point setting with absolute error bounds, we suppose
both arguments $x_1,x_2$ are represented with same length $n_1=n_2$
and same precision $m_1=m_2$: after possibly discarding trailing
significant digits and prepending leading \sdzero---possibly voiding previous normalization, though.
\item[f)]
Multiplication boils down 
(to adding the exponents and) multiplying the mantissae:
the latter as integers of (again supposed same) binary length $n+m$.
Note that $n+m\geq n'+m'\geq n+m-\calO(1)$ in Figure~\ref{f:FPCost}
means loosing at most constant many reliable bits;
similarly for other operations.
\item[g)]
Inversion $x\mapsto 1/x$ gets approximated by $\log(n+m)$ iterations 
$y_k\mapsto y_{k+1}:=(2-x \cdot y_k)\cdot y_k$,
starting with $y_0=3\cdot2^{-n-1}$ and
using addition and multiplication as above.
Recall that Newton-Raphson is \emph{stable} aka `self-correcting':
allowing the initial iterations to be conducted on the 
most significant digits only and then gradually increasing the
precision, such that the last iteration dominates
the overall cost up to a constant factor.
Regarding the resulting accuracy, 
\begin{equation} \label{e:Inverse1}
\big|\tfrac{1}{x+\varepsilon} \;-\; \tfrac{1}{x}\big| 
\;\; = \;\;
\frac{1}{|x|}\cdot \Big| \frac{1}{1-\delta} \;-\; 1\Big|
\;\;\leq\;\; 2|\delta|/|x| \;=\; 2|\varepsilon|/|x|^2
\end{equation}
provided $\delta:=\varepsilon/x\leq1/2$.
In particular, Estimate~\eqref{e:Inverse1} becomes $\leq2^{-m+2n'-1}$ 
when $|\varepsilon|\leq2^{-m}$ and $\tfrac{1}{|x|}\leq2^{n'}$ for $n'<m$. 
\item[h)]
Similarly, 
$\big|\sqrt{x+\varepsilon}-\sqrt{x}\big|=\sqrt{x}\cdot\big|\sqrt{1+\delta}-1\big|
\leq\sqrt{x}\cdot|\delta|=|\varepsilon|/\sqrt{x}$ 
when $|\delta|\leq1$ where $\delta:=\varepsilon/x$.
Computationally, Newton-Raphson iteration $y_k\mapsto y_{k+1}:=(3-x\cdot y_k^2)\cdot y_k/2$
converges quadratically to $1/\sqrt{x}$; and final multiplication by $x$ yields the result, cmp. \cite[\S5.3]{MullerElementary}.
\item[j)]
See \cite[Table~1]{Borwein2} with $M(n)=n\log n$
for the bit-cost of computing $\pi$.
\qed\end{enumerate}\end{remark}
Indeed, variable-precision arithmetic commonly underlies Reliable Numerics
as model of computation,
also in popular software and libraries \cite{MATLAB,MPFR,ARIADNE}:
like the integer register machine from Section~\ref{s:Introduction}
more convenient than, but both computably and complexity-theoretically 
equivalent to, the Turing model \cite{TPbook}.

\subsection{Quantitative Stability}
\label{ss:Stability}

When combining and iterating basic operations in user-controlled accuracy,
the programmer is responsible for keeping track of the 
(loss in) precision $m'$ evolving over the course of calculations:

\begin{example}[Logistic Map]
\label{x:Logistic}
For $3\leq r\leq 4$, consider the map 
\[ f_r:[0;1]\to[0;1], \; x\mapsto r\cdot x\cdot (1-x) \]
whose iterations are known to fluctuate `chaotically'.
Since $\max_x|f_r'(x)|=r$, an input error $\delta x$ can get
amplified by at most (and for $x=0$ up to) factor $r$;
by $r^k$ for the $k$-fold iterate $f_r^{(k)}$.
In terms of the above `binary' parameter $m\in\IZ$,
this amounts to $m'=m-\log_2 r$ and $m^{(k)}=m-k\log_2 r$.
\qed\end{example}
Such a mapping $m\mapsto m'$ captures 
Weierstra\ss{}' $\varepsilon$--$\delta$ continuity
in `binary' \cite[\S1.3.4]{Aras25}, 
namely as integer exponents to base $\tfrac{1}{2}$.
Alternative to determining the resulting precision $m'$ from given initial $m$,
applications may impose a desired result precision $m'$ and have the program/mer find a suitable initial $m$.
Such a mapping $\mu:m'\to m$ is known as
\emph{modulus} of continuity; see \cite[\S6]{Wei03} or
\cite[Definition~2.12]{Ko91} or \cite[Definition~6.2.6]{Wei00} or \cite[Example~4.8.2]{Koh08a}.
The literature often focuses on non-negative $m$ and $m'$, 
corresponding to $\delta=2^{-m}\ll1$ and $\varepsilon=2^{-m'}\ll1$;
recall Remark~\ref{r:FPCost}.

Beyond finding initial $m$ merely sufficiently large to yield the desired resulting $m'$,
$m$ should not be `unnecessarily' large either;
see also Remark~\ref{p:Stability2} below.
As formalized in Definition~\ref{d:Modulus},
the reference optimum $\mu_f:m'\mapsto m$ is intrinsic to the mathematical problem $f$ under consideration,
which the mapping $\mu_\calA:m'\mapsto m$ induced by any particular algorithm $\calA$ cannot beat. 
Note that, for particular values of $x$, $f_r^{(k)}$ from Example~\ref{x:Logistic} 
may behave better than in the worst-case $x=0$.
This additional dependence on $x$ is captured by a \emph{pointwise} 
(as opposed to uniform) modulus of continuity.
Note that, in order to thus improve the above worst-case estimate $r^k$
of the derivative of $f_r^{(k)}(x)$ for specific $x$,
the Chain Rule yields 
\[ \partial f_r^{(k)} \;=\; \prod\nolimits_{\ell=1}^{k} f'_r\circ f_r^{(\ell-1)} \]
which however is arguably even harder to evaluate than $f_r^{(k)}(x)$ itself.

\begin{definition}
\label{d:Modulus}
Let $\uobarZ=\IZ\cup\{-\infty,+\infty\}$
and write $\mu:\uobarZ\nearrow\uobarZ$ to indicate monotonicity.
\begin{enumerate}
\item[a)]
For a total function $f:X\to Y$ between metric spaces $(X,d)$ and $(Y,e)$,
its modulus of \emph{uniform} continuity is $\mu_f(m'):=\sup_{x\in X} \mu_f(x,m')$:
the pointwise least $\mu:\uobarZ\nearrow\uobarZ$ 
satisfying
$\forall x,x'\in X: d(x,x')\leq2^{-\mu(m')} \Rightarrow e\big(f(x),f(x')\big)\leq2^{-m'}$.
\item[b)]
For a function $f:\subseteq X\to Y$ 
(where ``$\subseteq$'' means partial),
its modulus of \emph{pointwise} continuity $\mu_f$
is the pointwise least mapping $\mu:\dom(f)\times\uobarZ\nearrow\uobarZ$ 
such that $d(x,x')\leq2^{-\mu(x,m')}$ for $x,x'\in\dom(f)$ implies $e\big(f(x),f(x')\big)\leq2^{-m'}$; 
with the convention $\mu(x,m')=+\infty$ at points $x$ of discontinuity,
and $\mu_f(x,m')=-\infty$ to mean 
$e\big(f(x),f(x')\big)\leq2^{-m'}$ regardless of $x'\in\dom(f)$
aka ``\emph{Don't care!}''
\item[c)]
To the operations from Figure~\ref{f:FPCost} 
correspond moduli $\mu:\IF\times\uobarZ\to\uobarZ$ as follows:
\[ 
\mu_{\pm}:m'\mapsto m'+1, \quad
\mu_{\times}:(n,m')\mapsto m'+2+n, \quad  \nonumber
\mu_{\scriptscriptstyle\sqrt{}}:m'\mapsto m'+1-\lceil n/2\rceil 
\] 
and $\mu_{\sign}:0\geq m'\mapsto-\infty$,  $1\leq m'\mapsto 1-n$,
compare Subsection~\ref{ss:ERC}a+d) below.
\end{enumerate}
\end{definition}
\begin{remark}
\label{r:Modulus}
\begin{enumerate}
\item[a)]
Moduli are non-decreasing and satisfy composition inequalities
\begin{equation}
\label{e:Composition}
\mu_{g\circ f}\;\leq\;\mu_f\circ\mu_g, \qquad
\mu_{g\circ f}(x,m')\;\leq\;\mu_f\big(x,\mu_g(f(x),m')\big) \enspace . 
\end{equation}
\item[b)]
By Definition~\ref{d:Modulus},
the pointwise modulus $\mu_f(x,n)$ 
depends on $x$ exactly.
However triangle inequality yields
$\mu_f(x',n)\leq \mu_f(x,n+1)+1$ whenever $d(x,x')\leq 2^{-\mu_f(x,n+1)-1}$.
\item[c)]
A function $f:[0;1]\to\IC$ is H\"{o}lder-continuous
~iff~ it has uniform modulus $\mu_f(n)\leq \calO(n)$ for $n\geq0$;
$f$ is Lipschitz-continuous
~iff~ it has uniform modulus $\mu_f(n)\leq n+\calO(1)$ for $n\geq0$.
\item[d)]
The function 
\[ h:[0;1]\to[0;1], \quad 0<x\mapsto 1/\ln(e/x), \quad 0\mapsto 0 \]
is continuous with exponential modulus of uniform continuity,
which is also its modulus of pointwise continuity at $x=0$:
$\mu_h(n)=\lceil (2^n-1)/\ln2\rceil$ for $n\geq0$.
\\
The composition $h\circ h$ has doubly exponential
modulus of uniform continuity,
which is also its modulus of pointwise continuity at $x=0$.
\item[e)]
Refining Remark~\ref{r:MainThm},
computing a function $f$ with modulus $\mu_f$
on a Turing machine requires time $\mu_f(n+2)$ \cite[Theorem~2.19]{Ko91}.
\end{enumerate}
\end{remark}

\subsection{Example Cost Analyses}
\label{ss:Examples}

The global bit-cost of some integer computation amounts to the sum of the local costs,
where each operation is weighted in dependence on the length $n$ of its operand/s.
Similarly the global bit-cost of an approximate real number computation
amounts to adding the local costs of the operations executed,
each one weighted in dependence on both 
the length $n$ and precision $m$ of its operand/s.

\begin{example}[Logistic Map, Continued]
\label{x:Logistic2}
The following trivial algorithm $\calA$, 
calculating the $k$-fold iterated Logistic Map \naively,
has global modulus $\mu_{\calA}(p)=p+\calO(k)$
and is thus almost optimal w.r.t. $p$:
    \begin{algorithmic}[0] 
    \Prog{$r,x:[0;1]$;\; $k:\IN$}
    \While{$k>0$}
    \State $x\gets r\cdot x\cdot (1-x)$;
    \State $k\gets k-1$;
    \EndWhile
    \EndProg{$x$}
    \end{algorithmic}
Indeed, each multiplication in $[0;1]$ at working precision $m$
incurs local costs $\calO(m\log m)$ and each iteration
`looses' $\calO(1)$ bits of precision
according to Figure~\ref{f:FPCost}.
To attain output error $\leq2^{-p}$, 
the global bit-cost therefore amounts to $\calO\big(k(p+k)\log(p+k)\big)$;
as experimentally confirmed in Appendix~\ref{a:Experiments}.
\end{example}

\begin{example}[Determinant]
\label{x:Determinant}
Consider the problem of calculating the determinant of a given
$d\times d$ matrix $A$ with entries in $[-1;1]$; hence $n=\calO(1)$.
\begin{enumerate}
\item[a)]
Leibniz' formula 
\begin{equation}
\label{e:Leibniz}
\det(A) \;=\; \sum\nolimits_{\sigma\in\calS_d} \sign(\sigma)\cdot a_{1,\sigma(1)}\cdot a_{2,\sigma(2)}\cdots a_{d,\sigma(d)} 
\end{equation}
shows that
$|\partial\det(A)/\partial a_{ij}|\leq (d-1)!$;
hence $\mu_{\det}(p)\leq p+\calO(d\log d)$.
\item[b)]
When implementing Leibniz' Formula,
each multiplication by the next $a_{j,\sigma(j)}$ looses $\calO(1)$ bits of precision;
and so does each addition for the next $\sigma\in\calS$.
The latter $d!$ terms are therefore better not added iteratively,
but instead pairwise bottom-up in a binary tree of depth $\calO(d\log d)$:
This yields an algorithm $\calA$ for the restriction
of $\det$ to $[-1;1]^{d\times d}$ with modulus $\mu_{\calA}(p)=p+\calO(d\log d)$
and global bit-cost $\calO\big((p+d\log d)\cdot d!\big)$:
optimal (a) w.r.t. $p$ but exponential in $d$.
\item[c)]
A numerically more common approach, using $\calO(d^3)$ arithmetic operations and tests,
first transforms $A$ in---not necessarily reduced---row echelon form
(Gaussian Elimination) and then multiplies the diagonal entries.
However this algorithm $\calB$, thus accelerated w.r.t. $d$
and even restricted to compact $[-1;1]^d$,
has no \emph{global} modulus:
The intermediate row echelon form depends \emph{dis}continuously on the matrix entries.
\begin{equation}
\label{e:RowEchelon}
\binom{0 \;\; 1}{0 \;\; 0}   
\quad\surd 
\qquad\text{vs.}\qquad 
\binom{0 \;\; 1}{\varepsilon \;\; 0}   
\quad 
\neovsearrow\quad\binom{\varepsilon \;\; 0}{0 \;\; 1} \quad\surd
\qquad (\varepsilon\neq0) 
\end{equation}
\item[d)]
Algorithm $\calB$ from (c) is discontinuous
when considered operating on real numbers \emph{extensionally}.
However considered \emph{in}tensionally, namely operating on approximations to $A$ up to error $2^{-m}$,
$\calB$ does produce approximations to $\det(A)$ up to error $2^{-p}$
for $p=m-\calO(d\log d)$ according to (a).
Its global bit-cost amounts to $\calO\big((p+d\log d)\cdot d^3\big)$.
\end{enumerate}
\end{example}

\subsection{Exact Real Computation}
\label{ss:ERC}

The \emph{Exact Real Computation} (\ERC) paradigm relieves the user programmer 
from worrying about the underlying discretization and error propagation,
for continuous data types such as real numbers. 
Similarly to \texttt{BigInteger} coinciding with $\IZ$, \ERC provides a basic data type
\texttt{REAL} coinciding with $\IR$ 
including transcendental numbers and operations.

{
\begin{description}
\item[a)]
Software implementing \ERC
ensures that each data item have precision (finite but) 
sufficient for the user program to execute as if on exact data.
In particular real test ``$r>0$?'' tacitly approximates 
$r\in\IR$ sufficiently well to determine its sign,
namely up to error $|r|/2$: 
thus $\mu_{\sign}$ in Definition~\ref{d:Modulus}c)
and condition $m>-n$ in Definition~\ref{d:Contribution}.
\item[b)] 
A working precision $2^{-m}$ sufficient for
both, successful intermediate sign tests and guaranteed output error $\leq2^{-p}$,
can be determined in various ways:
by maintaining symbolic representations of the user calculations \cite{Lam07}; 
or by transparently re-executing user code in higher and higher precision,
until the resulting error is small enough \cite{Mue01};
or by some combination of both.
See Problem~\ref{p:Stability2} below, 
and also \cite{GL01,Men05} for further implementations of \ERC.
\item[c)]
Note that (b) means taking over control from the user program, 
namely changing its \naive command execution order!
This is common in declarative programming
but violates the implicit imperative conceptions 
usually associated with \texttt{C++}.
Experienced imperative coders new to \ERC usually experience some learning curve, 
for instance regarding debugging outputs or `premature optimizations' (Don Knuth):
A program's runtime may seem spent in a user loop, 
while actually it arises elsewhere---such as 
when some earlier lazy real test gets evaluated 
and turns out as (almost) equal; see (d\,e\,f) below.
Nevertheless, \ERC can be considered imperative,
although with multivalued semantics \cite[\S4]{PBC+24b};
and the present work shows that \ERC even allows for bit cost analyses!
Intuitively, non-imperative `behavior' 
(of a program $\calA$ on input $\bar x$ and for output precision parameter $p$, say)
arises from determining a suitable working precision parameter $m$ according to (b);
and if/once that has been found somehow, all commands and tests and loops 
can be simulated in their user-intended imperative order.
Subsection~\ref{ss:Stability2} below formalizes and promotes
such integer mappings $(\calA,\bar x,p)\mapsto m$ 
as numerical algorithmic problem in its own right.
\item[d)]
Determining $\sign(r)$ from sufficiently accurate approximations 
to $r$ as in (a) is feasible when $r\neq0$
(possibly after re-calculation/s in higher working precision),
but inherently infeasible when $r=0$:
in agreement with numerical folklore and with Remark~\ref{r:FPCost}d):
Real sign tests in \ERC are \emph{partial}, or
total with \emph{lazy} `values' in Kleene's trinary logic;
see Example~\ref{x:Kleene} below. 
\item[e)] 
Accordingly, two-fold total \texttt{if-then-else} branching on one Boolean condition
is replaced in \ERC by a \texttt{switch} aka $\mychoose()$
on two or more `unevaluated' (=\emph{lazy}) Kleenean arguments;
with the promise that at least one of the arguments (is defined and) evaluates to \true, 
and the user program must not rely on which \texttt{case} gets executed
when several arguments are \true: 
leading to a \emph{multivalued} semantics \cite{Luc77};
see Example~\ref{x:Choose} below.
\item[f)]
In a nutshell, \ERC exhibits the following features and limitations:
real (including transcendental) numbers are a basic data type 
which can be added and multiplied and divided exactly;
real numbers can be compared, as partial operation with lazy evaluation (recall c).
Integers can be added/subtracted/compared and divided by 2
\cite[\S2.6]{Papadimitriou}, 
but not multiplied---to obtain a two-sorted structure 
with decidable first-order theory \cite[Theorem~6.2]{PBC+24b}.
For the same reason, integers cannot be converted/cast to reals directly
but only as exponent to base 2, useful for error bounds: 
$\error:\IZ\ni z\mapsto 2^z\in\IR$; see also Definition~\ref{d:Contribution} below.
Finally transcendental constants, and functions,
can be expressed by means of an \emph{implicit} Cauchy limit functional
explained next:
\item[g)]
For simplicity, in the present work, 
\ERC is considered strictly first-order \cite[\S1.3]{SICP}:
A program has a constant number of variables;
higher types, such as arrays (indirect addressing), are banned;
functions appear as \emph{second} class `citizens' only.
More precisely an \ERC program \emph{realizes} one \texttt{main()}
and possibly further user functions which may invoke (``call'') each other.
Invoking a real such function $f$ with arguments $\vec x\in\dom(f)$ 
returns its value $y=f(\vec x)$ exactly.
To this end, an \ERC program $\calA$ realizing $f$ receives,
in addition to arguments $\vec x$, a parameter $p\in\IN$;
and must return some approximation $\tilde y_p=\calA(\vec x,p)$ to $y$ up to error $2^{-p}$:
Proceeding to the limit $f(\vec x)=\lim_{p\to\infty} \calA(\vec x,p)$ 
is handled implicitly by the \ERC environment,
see Example~\ref{x:Limit} below.
\end{description}}

\noindent
\cite{PBC+24b} captures the formal syntax and semantics of \ERC
(just with $-p\in\IZ$ instead of $p\in\IN$), and
illustrates programming with nine examples \cite[\S5.2]{PBC+24b}
including the real exponential function.
Appendix~\ref{a:Steep} elaborates another example: 
realizing in \ERC the `steep' function $h$ from Remark~\ref{r:Modulus}d).
Section~\ref{s:ERC2} below formalizes 
\ERC as \emph{Computation Forests} with associated moduli of continuity:
one Computation Tree for each user function including \emph{Main}.

\cite[\S5.3]{PBC+24b} proves that the 
computational power of \ERC
coincides with that of 
Turing Machines in Computable Analysis .
In the sequel we quantitatively strengthen this qualitative statement:
asserting that \emph{real} polynomial-time computations on a Turing Machine correspond to
algorithms in \ERC of polynomial cost, and vice versa.
To this end we first assign local bit-costs to the primitive operations in \ERC,
and then to entire program executions.

\begin{remark}
\label{r:iRRAM}
The \ERC paradigm is heavily inspired by,
and establishes a formal (simplified/idealized) programming language foundation to,
the \textsf{iRRAM} library \cite{Mue01}; which in turn builds on \cite{BH98}.
For instance \textsf{iRRAM} itself, based on \texttt{C++},
does support integer multiplication and dynamic arrays and operators on real functions 
as first-class objects---including an \emph{explicit} limit.
See \cite{Clerical} for a variant of \ERC closer to \textsf{iRRAM} practice;
and \cite{realLib,core2,Aern,ARIADNE} for other implementations (in the spirit) of \ERC.
\end{remark}

\section{Bit-Cost in Exact Real Computation}
\label{s:Contribution}

The bit-cost of an integer operation depends on the length $n$ of its arguments;
and the bit-costs of approximate real number operations 
depend on both magnitude $n$ and working precision $m$.
Figures~\ref{f:IntCost} and \ref{f:FPCost} motivate assigning
\emph{local} bit-costs in \ERC as follows:

\begin{definition}
\label{d:Contribution}
To each basic operation $\omicron$ supported by \ERC (as in Subsection~\ref{ss:ERC}),
assign \emph{local} $\cost(\omicron)$ up to constant factors/offsets as follows, in dependence on
the (binary) magnitude parameter $n\geq\myn(y)=1+\lfloor\log_2|y|\rfloor\in\IZ$ of operand $y$ and---for 
real arguments---also on the working precision parameter $m\in\IZ$. 
Condition $m>-n$ captures that at least one digit must be reliable; 
and $\alpha:=\log\max\{n,m\}$ is charged to access the least/most significant digit.
\begin{center}
\rm\begin{tabular}{c@{\;}|@{\;}c|@{\;}c}
\textbf{operation} & \textbf{bit-cost}$(n,m)$ & \textbf{comment} 
\\ \hline
\makecell{adding/subtracting \\[-0.3ex]
comparing/halving integers} & $n$ & Fig~\ref{f:IntCost} \\[+0.5ex]
adding/subtracting two reals	& $n+m\;+\alpha$ & Remark~\ref{r:FPCost}e) \\
multiplying two reals		& $(n+m)\cdot\log(n+m)+\alpha$  & Remark~\ref{r:FPCost}f) \\
inverting a real 		& $(n+m)\cdot\log(n+m)+\alpha$  & Remark~\ref{r:FPCost}g) \\
real test $\sign(y)$		& $\max\{-n,1\}\;+X$ & Subsection~\ref{ss:ERC}a) \\[+0.3ex]  
$\error:\IZ\ni z\mapsto 2^z\in\IR$ & $n$ & \makecell{$n'=2^n, m'=\infty$ \\[-0.3ex] Subsection~\ref{ss:ERC}f} \\[+0.5ex]
\mychoose$(\tau_0,\ldots,\tau_{K-1})$ & $\min\{\cost(\tau_k):\tau_k=\true\}$ & Example~\ref{x:Choose} \\
$\lim_k y_k$ & $\cost(y_{m+1},m+1)$ & Example~\ref{x:Limit}  \\
call $f(\vec y)$ & $\cost\big(\calA,\vec y,m\big)$ & Subsection~\ref{ss:ERC}g\qed
\end{tabular}
\end{center}
\end{definition}
The above local cost assignments 
depend on the intermediate values' magnitude $n$ and working precision $m$, 
both of which evolve throughout the computation;
see Example~\ref{x:C}d) above and Remark~\ref{r:Local} below.
Similar to high-level integer programming languages (Section~\ref{s:Introduction}),
and as opposed to user-controlled accuracy calculations (Subsection~\ref{ss:FPCost}),
\ERC relieves the programmer from maintaining suitable $n$ and $m$;
but cost analyses still need to keep track of (now both of) them.

\begin{remark}
\label{r:Local} 
\begin{enumerate}
\item[a)]
Recall from the integer register machine (Section~\ref{s:Introduction}) that the intermediate length parameter $n$ is initially determined
by the argument $\vec x$ passed to the program, and then naturally
evolves with (and, reflecting the imperative paradigm, 
in the order of) its operations being executed---basically by forward induction on the number $s$ of operations: 
$n_0:=\myn(\vec x)=\max\big\{d,\myn(x_1)+\cdots+\myn(x_d)\big\}$ and $n_{s+1}:=\myn_{\omicron_s}(n_s)$,
where $\omicron_s$ denotes the $s$-th primitive performed,
and $\myn_{\omicron}$ captures the induced growth in binary length
from argument to return value.
\item[b)]
The working precision parameter $m$ in \ERC on the other hand evolves `backwards':
such as to guarantee the given \emph{output} error bound $2^{-p}$;
and for all real sign tests executed along the way 
to succeed according to Subsection~\ref{ss:ERC}a).
Subsection~\ref{ss:Stability2} formalizes this backward evolution of working precision.
\item[c)]
Recall that, for integer register machines,
generic polynomial cost (ii\myprime in Remark~\ref{r:ChurchTuring})
is equivalent to both being polynomially bounded:
the number $N$ of commands executed
and the worst-case length $n$, 
namely of the largest intermediate value arising.
\end{enumerate}
\noindent
Similarly, as captured in Definition~\ref{d:ERC}c) below,
\emph{polynomial cost in \ERC} (iv\myprime)
means that all three are polynomially bounded:
number of commands, length of the largest value,
and working precision $m$ sufficient for  
real sign tests to succeed and for guaranteed output error $\leq2^{-p}$.%
\end{remark}
Regarding (b), the modulus from Definition~\ref{d:Modulus} already formalizes 
backwards dependence of input precision on output precision: 
for single, unary functions.
Subsection~\ref{ss:Modulus2} below refines and extends 
said modulus to multivariate partial functions;
and Subsection~\ref{ss:Choose} to generalized predicates and multifunctions, 
like $\mychoose()$.
Subsection~\ref{ss:Round} applies these in order to analyze
two example \ERC algorithms for multivalued integer rounding:
one incurring cost exponential and the other polynomial in $n+m$.
The latter one enters crucially in Subsection~\ref{ss:Simul1}
for the efficient simulation of real Turing computation in \ERC.
The converse is treated in Section~\ref{s:ERC2}.

\subsection{Quantitative Stability of Algorithms}
\label{ss:Stability2}

Stability analysis is a core task in Numerics \cite{Condition}:
stability of both functions and algorithms,
in the sense of how much an input perturbation affects the output,
quantitatively captured by moduli of continuity.
Naturally a algorithm $\calA$ computing some function $f$
cannot have better stability than $f$ itself;
and one may want to design $\calA$ with modulus close to that intrinsic to $f$.
A `small' modulus means that calculations can be performed with `few' bits,
which tends to reduce the computational cost per operation---although $\calA'$ with
smaller modulus might in turn need more operations,  
at possibly a trade-off.
Subsection~\ref{ss:Examples} for instance has manually estimated the loss of precision incurred by
some numerical example algorithms, and has compared that to the function they compute.
Reliable Numerics, such as interval arithmetic, 
considers computations that enhance their approximate output
with guaranteed error bounds in dependence on the input precision,
automatically keeping track of rounding errors and error propagation 
on behalf of the user. 
The present subsection is concerned with the following setting:

\begin{problem}
\label{p:Stability2}
Given a reliable numerical algorithm $\calA$ and continuous-type input $\bar x$
and an integer $p$ describing the desired output error bound $\leq2^{-p}$:
determine an integer error parameter $m$ for perturbed inputs $\bar x'=\bar x\pm2^{-m}$
that guarantees $\calA$ to produce outputs $y$ within error $\leq2^{-p}$;
recall Subsection~\ref{ss:ERC}a).

Such a modulus $\mu_\calA(\bar x,p)\mapsto m$ thus must be conservative.
One may additionally ask it to be `close' (\eg up to polynomial
or up to constant factors or offsets) to optimal, namely to the
least such mapping associated with $\calA$.
\qed\end{problem}
\noindent
Both problems are well-known in Reliable Computing \cite{VANDERHOEVEN200652}, 
although often informally with emphasis on practice. 
The so-called \emph{bottom-up} method starts with some initial candidate error parameter $m$,
executes $\calA$ and records the resulting result error parameter $m'$
and, if $m'<p$, restarts with larger $m$.
The \emph{top-down} method simulates $\calA$ once
and records its calculations symbolically, such as in a directed acyclic graph (dag); 
then promotes the accuracy requirements from root ($\leq2^{-p}$) 
on downwards, level by level, until reaching the children/inputs
when a desired $m$ can be read off:
Compare \cite[\S2.2]{realLib} or \cite[\S3.4]{core2}.

Section~\ref{s:ERC2} below formalizes the common core idea behind both approaches:
understanding an algorithm $\calA$ as a large composition of computational primitives,
each one associated with a modulus as in Definition~\ref{d:Modulus}c);
and combining those moduli of the operations executed during a run of $\calA$ on input $\bar x$,
in \emph{reverse} composition according to Estimate~\eqref{e:Composition2}. 
Specifically Subsection~\ref{ss:SLP} thus assigns moduli
to a suitable class of Straight-Line Programs (without branches);
Subsection~\ref{ss:CT} to computation Trees (with non-deterministic branches),
and Subsection~\ref{ss:CF} to computation forests (with mutual function calls).

\begin{remark}
\label{r:OverEstimate}
Questions about composite calculations, combined from individual simple operations, can be notoriously 
difficult already in the discrete/integer realm: recall Rice's or Goodstein's Theorems.
The map $x\mapsto x(1-x)$, such
as in Example~\ref{x:Logistic2} for $r=1$, has values bounded by $1/4$ for $x\in[0;1]$;
but intermediate expression $y:=1-x$ can exhaust entire $[0;1]$, 
and so does the product $xy$---when disregarding the correlation between $x$ and $y$.

Numerical error propagation estimates are well-known to suffer 
from similar overestimates \cite{Neumaier1993}, 
also intrinsically \cite{AllenderBKM09}.
One approach to mitigate such (variants of essentially the same) problems
combines several basic operations into a 
more involved so-called \emph{Taylor Model};
whose bounds and stability are analyzed manually 
and told to (rather than estimated by) 
the computer \cite{TaylorModels}.
\end{remark}
The rest of this section refines and applies
the modulus of continuity for
backwards error propagation estimates
according to Remark~\ref{r:Modulus}a).

\subsection{Advanced Quantitative Continuity}
\label{ss:Modulus2}

This subsection refines the well-known modulus of continuity (Definition~\ref{d:Modulus})
in two ways: 
from univariate to multivariate functions, 
and to (a small class of) multivalued functions.

In case the continuous function $f$ is $K$-variate, 
Definition~\ref{d:Modulus} applies with respect to some (implicit) metric on the product space
$X=X_1\times\cdots\times X_K$. 
Alternatively we explicitly consider pointwise moduli to be of type
$\mu:\subseteq\big(\prod\nolimits_{k=1}^K X_k\big)\times\uobarZ\nearrow\uobarZ^K$,
\begin{equation}
\label{e:Multidim}
\bigwedge\limits_{k=1}^K d_k(x_k,x'_k)\leq2^{-\mu_{f,k}(x_1,\ldots,x_K,m')} 
\;\;\Rightarrow\;
e\big(f(x_1,\ldots,x_K),f(x'_1,\ldots,x'_K)\big)\leq2^{-m'}
\end{equation}
and similarly in the uniform case $\mu:\uobarZ\nearrow\uobarZ^K$.
Projection $(x_1,\ldots,x_K)\mapsto x_k$ for instance
has modulus 
\[ \mu:m'\mapsto (-\infty,\ldots,-\infty,m',-\infty,\ldots,-\infty) \enspace ; \]
recall Convention~\ref{d:Modulus}a).

\begin{remark}
\label{r:Composition}
For bivariate functions $f:X\times Y\to X'$ and $g:X\times Y\to Y'$ and $h:X'\times Y'\to Z$,
Estimates~\eqref{e:Composition} become, componentwise, 
\begin{eqnarray}
\mu_{h\circ(f,g)} &\leq& \big(
\max\{\mu_{f,1}\circ\mu_{h,1},\mu_{g,1}\circ\mu_{h,2}\}
\:,
\max\{\mu_{f,2}\circ\mu_{h,1},\mu_{g,2}\circ\mu_{h,2}\}
\big) \nonumber \\
\label{e:Composition2}
\mu_{h\circ(f,g)}(x,y,m') &\leq& 
\Big( \max\big\{ 
\mu_{f,1}\big(x,y,\mu_{h,1}(f(x,y),g(x,y),m')\big), \\ \nonumber
& & \qquad\quad
\mu_{g,1}\big(x,y,\mu_{h,2}(f(x,y),g(x,y),m')\big)\big\} 
\:,\:  \\ \nonumber & & \quad \max\big\{
\mu_{f,2}\big(x,y,\mu_{h,1}(f(x,y),g(x,y),m')\big), \\ \nonumber
& & \qquad\quad\:
\mu_{g,2}\big(x,y,\mu_{h,2}(f(x,y),g(x,y),m')\big)\big\} 
\Big) \enspace .  \nonumber
\end{eqnarray}
\end{remark}
A multivariate continuous function $f$ need not have 
a least modulus as in Definition~\ref{d:Modulus}, though.
Instead of `the' modulus, one thus has to refer to \emph{a} modulus of $f$;
and \emph{least} modulus becomes \emph{Pareto}-optimal moduli:
pairwise incomparable and least in at least one component,
such as $\mu_\vee(\true,\true,1)$ and $\mu_\wedge(\false,\false,1)$ below.

\begin{example}[Kleenean Connectives]
\label{x:Kleene}
Recall Kleene's trinary logic \cite[p.336]{Kleene52}
with values $\IK=\{\false, \true, \unknown\}$,
equipped with the discrete metric.
In view of the absorption laws,
the binary Boolean connectives $\vee,\wedge$
have the following moduli of pointwise continuity:
\begin{equation}
\label{e:Kleene}
\begin{array}{c|ccc}
\mu_\vee & \false & \true & \unknown
\\ \hline
\false & (1,1) & (-\infty,1) & (1,1) \\[0.7ex]
\true & (1,-\infty) & {(1,-\infty)}\atop{(-\infty,1)} & (1,-\infty) \\[0.7ex]
\unknown & (1,1) & (-\infty,1) & (1,1) 
\end{array} 
\qquad
\begin{array}{c|ccc}
\mu_\wedge & \false & \true & \unknown
\\ \hline 
\false & {(1,-\infty)}\atop{(-\infty,1)} & (1,-\infty) & (1,-\infty) \\[0.7ex]
\true & (-\infty,1) & (1,1) & (1,1) \\[0.7ex]
\unknown & (-\infty,1) & (1,1) & (1,1) 
\end{array}
\end{equation}
where $\mu_\vee,\mu_\wedge$ means on argument $(\text{row},\text{column},m')$
for any $m'\geq1$; $\mu_\vee(x,y,m'),$\linebreak$\mu_\wedge(x,y,m')=(-\infty,-\infty)$ for $m'\leq0$.
Also $\mu_\neg(x,m')=1$ for $m'\geq1$ and $\mu_\neg(x,m')=-\infty$ for $m'\leq0$.
\end{example}

\begin{example}[Modulus of Fast Cauchy Limit]
\label{x:Limit}
The field of polynomial-time computable real numbers 
is closed under `fast' Cauchy limits \cite{Mue86a}:
Whenever a sequence $\bar y=(y_k)\in\IR^\IN$ is uniformly computable in polynomial time,
and if it furthermore satisfies 
\begin{equation}
\label{e:Cauchy}
\forall k,\ell\in\IN: \; |y_k-y_\ell| \;\leq\; 2^{-k}+2^{-\ell}
\enspace ,
\end{equation}
then $y=\lim_k y_k$ is again computable in polynomial time.
The subset of $\IR^\IN$ according to Condition~\eqref{e:Cauchy} is topologically closed; 
and on this domain,
said polynomial-time computable limit functional 
has the following family of Pareto-optimal 
moduli of continuity $\mu_{\lim}:\IZ\to\IZ^\IN$ independent of $\bar y\in\dom(\lim)$:
\begin{eqnarray}
\IN\ni m' \;\mapsto\; (-\infty,\cdots,-\infty,&\overset{m\text{-th}}{\pmb{+}\infty}&,{-\infty},-\infty,-\infty,\ldots) \nonumber \\
\IN\ni m' \;\mapsto\; (-\infty,\cdots,-\infty,&\pmb{-}\infty&,m'\!+\!1,-\infty,-\infty,\ldots) 
\label{e:Limit} \\ \nonumber
\IN\ni m' \;\mapsto\; (-\infty,\cdots,-\infty,&\pmb{-}\infty&,-\infty,m'\!+\!1,-\infty,\ldots) \\
\nonumber & & \qquad\qquad \ddots 
\end{eqnarray}
and 
$0>-m'\mapsto (-m'+1,-\infty,-\infty,\ldots)$;
see Subsection~\ref{ss:ERC}g).
\end{example}
%

\subsection{Generalized Predicates}
\label{ss:Choose}

Classical predicates formalize two-valued total tests.
On a connected domain, any such test must be discontinuous---and hence uncomputable,
recall Remark~\ref{r:MainThm}.
One remedy considers \emph{partial} tests, such as $\sign:\IR\setminus\{0\}\to\{\true,\false\}$.
Note the subtle but important difference to a total test, with Kleenean value $\sign(0)=\unknown$ from Example~\ref{x:Kleene}.
In order to use said partial tests in a total program,
\ERC provides the \mychoose multifunction
which receives as arguments several `overlapping' unevaluated (\ie \emph{lazy}) tests;
see Example~\ref{x:Choose} below.
Formally, a multifunction may violate the extensionality axiom $x=x'\Rightarrow f(x)=f(x')$.
These are well-known unavoidable in Computable Analysis \cite{Luc77},
where extensionality requires that different encodings of the same real argument yield (possibly different encodings, but of) the same value.
Notions of continuity for relations common in Analysis, such 
as \emph{hemi}continuity, tend to be too strong or too weak 
to capture this computational phenomenon qualitatively \cite{PZ13},
not to mention quantitatively. 
The following formalization restricts to piecewise constant discretely-valued multifunctions,
capturing partial non-extensional (``non-deterministic'') predicates 
with possibly more than two outcomes---and 
in particular the \mychoose operation from \ERC
\cite[Examples~2.2(3)]{PBC+24b}; see Example~\ref{x:Choose} below.

\begin{definition}
\label{d:MultiPred}
Fix an at most countable (and at least two-element) 
set $\IT=\{\false,\true,\ldots\}$ 
of generalized truth values,
and a metric space $(X,d)$. 
A \emph{generalized} $T$-fold predicate on $X$
is a subset $P\subseteq X\times\IT$, where $T=\Card(\IT)$.
For $x\in X$, $P(x):=\{\tau:(x,\tau)\in P\}$ denotes the (possibly empty) set of truth values 
of predicate $P$ on $x$; written $x\mapstoto P(x)$.
Its domain is $\dom(P)=\{x:P(x)\neq\emptyset\}$;
$P^{-1}[\tau]=\{x:\tau\in P(x)\}$ the set of arguments 
that \emph{may} produce truth value $\tau\in\IT$.
\begin{enumerate}
\item[a)]
Call $P$ \emph{pointwise continuous} if it holds
\begin{equation}
\label{e:MultiCont}
\forall x\in\dom(P) \; \exists \delta>0 \; \exists \tau\in P(x) \; 
\forall x'\in\dom(P): \quad d(x,x')\leq\delta \;\Rightarrow \; \tau\in P(x') \enspace . 
\end{equation}
\item[b)]
A modulus of \emph{pointwise} continuity of $P$ is $\mu:\dom(P)\times\uobarZ\nearrow\uobarZ$
such that $\delta:=2^{-\mu(x,k)}$ satisfies Equation~\eqref{e:MultiCont}.
\item[c)]
Similarly $\mu:\dom(P)\times\uobarZ\nearrow\uobarZ^K$ for $K$-variate $P$,
\ie, in case $X=X_1\times\cdots\times X_K$ is a product of metric spaces:
\begin{multline*}
\forall \vec x=(x_1,\ldots,x_K)\in\dom(P) \;\; \exists \tau\in P(\vec x) \;\; 
\forall \vec x'=(x'_1,\ldots,x'_K)\in\dom(P): \\
\bigwedge\nolimits_{k=1}^K d(x_k,x'_k)\leq2^{-\mu_k(\vec x,k)} 
\;\;\Rightarrow \; \tau\in P(\vec x') \enspace . 
\end{multline*}
\end{enumerate}
\end{definition}
As mentioned after Remark~\ref{r:Composition}, a continuous multivariate function
may not have an optimal pointwise modulus---nor a uniform modulus,
even when restricting to compact domains.

\begin{example}
\label{x:Choose}
Different from \ERC's partial real test,
recall \cite[\S6]{Soft} the so-called \emph{soft} or \emph{fuzzy} test
$\sign_\varepsilon(x)$: 
allowed to return $\true$ whenever $x\geq-\varepsilon$ and to return $\false$ whenever $x\leq+\varepsilon$.
Intuitively, the `overlap' in case $-\varepsilon\leq x\leq+\varepsilon$ 
makes $\sign_\varepsilon$ total,
but non-extensional a multifunction as above.
It can be expressed as
$\mychoose(x<+\varepsilon,x>-\varepsilon)$, 
with \ERC's built-in multifunction
\begin{equation}
\label{e:Choose}
\mychoose(\tau_0,\ldots,\tau_{K-1}) \;=\; \text{``some'' } k \text{ s.t. } \tau_k=\true
\end{equation}
where $(\tau_0,\ldots,\tau_{K-1})\in\IK^K$ is promised to satisfy $\exists k: \tau_k=\true$;
again note the multivaluedness when more than one $k$ satisfies $\tau_k=\true$.
It has Pareto moduli 
\[
\dom(\mychoose)\times\IN_+ \;\ni\; (\vec\tau,m') \;\mapsto\; (-\infty,\cdots,-\infty,\overset{k\text{-th}}{1},-\infty,\ldots) 
\]
for $m'\geq1$ and any $k$ with $\tau_k=\true$.
Indeed \ERC internally is supposed to search $k$ 
for which $\tau_k=\true$ is \emph{cheapest}
to evaluate: thus the $\min$ (not $\max$)
in Definition~\ref{d:Contribution}.
\end{example}
%

\subsection{Examples: Multivalued Integer Rounding}
\label{ss:Round}

Consider the problem of rounding a given real to a near integer.
A non-constant function 
mapping the connected real line to the discrete integers cannot be continuous,
nor computable (Remark~\ref{r:MainThm}).
Hence \cite[\S5.2.4]{PBC+24b} considers rounding as a multifunction,
recall 
Subsection~\ref{ss:Choose}:
\begin{equation}
\label{e:Round}
\Round \;:\; \IR \;\ni\; x \;\mapstoto \; k \;\in\;\IZ, \quad x-1<k<x+1
\enspace .
\end{equation}
So both $0$ and $1$ (and none else) are possible return values 
when rounding $x=1/2$, or any $x\in(0;1)$ for that matter.

\cite[\S5.2]{PBC+24b} demonstrates 
programming in \ERC (Subsection~\ref{ss:ERC}f+g)
with nine `toy' examples, two of them (\#5 and \#6) 
both calculating $\Round$ according to Equation~\eqref{e:Round}
employing the generalized predicate \texttt{choose()}:
the first one (Example~\ref{x:URound}) with cost softly linear,
the second one (Example~\ref{x:BRound}) with cost polylogarithmic,
in the value of the real argument---which in turn is commonly considered exponential in its length.
Recall (Subsection~\ref{ss:ERC}) that some implementations of \ERC
might, while others might not, simply read off the integer part from their
internal real representation; but such implementation details
and dependencies should be hidden from the high-level user.

\begin{example}
\label{x:URound}
Algorithm~\#5 in \cite[\S5.2]{PBC+24b} initializes $k$ to zero;
then decrements $x$ and simultaneously increments $k$
repeatedly, until $x$ becomes negative.
(We focus here on the case of positive arguments;
negative arguments dually proceed by repeatedly
incrementing $x$ and decrementing $k$\ldots)
More precisely, after each iteration, it employs \texttt{choose()}
to simultaneously perform partial tests $x<1$ and $x>1/2$;
and abort in case the first test evaluates to \true
and continue in case the second test evaluates to \true
and non-deterministically either abort or continue
in case both evaluate to \true:
recall Subsection~\ref{ss:Choose}. 
Thus the loop gets executed (and $k$ incremented and returned as result) 
either $\lfloor x\rfloor$ times or $1+\lfloor x\rfloor$ times:
in accordance with Specification~\eqref{e:Round}.
According to Definition~\ref{d:Contribution},
each loop iteration incurs bit-cost $\calO(\log|x|)$
for decrementing $x$; $\calO(\log|k|)\leq\calO(\log|x|)$ 
for incrementing $k$.
And while tests ``$x<1$?'' and ``$x>1/2$?'' may be unboundedly
expensive for $x\approx1$ and $x\approx1/2$, respectively;
the cost incurred by them wrapped in \texttt{choose()} 
never exceeds $\calO(\log|x|)$ according to Definition~\ref{d:Contribution}:
adding up to total cost $\calO(|x|\cdot\log|x|)$.
\end{example}
\begin{example}
\label{x:BRound}
Recall that binary search for the, say, leading bit of $x$ 
is (unique unless dyadic and) discontinuous and uncomputable!
Instead, Algorithm~\#6 in \cite[\S5.2]{PBC+24b} 
extracts the highly non-unique \emph{signed} binary expansion
of $y\in(-1;1)$: using a three-fold combined test with overlap
\[ \mychoose(y<0\:,\: -1<y \wedge y<1 \:,\: 0<y) \;-\; 1 \enspace . \]
Spelled out, it returns leading digit $b:=-1$ in case $y<0$;
digit $b:=0$ in case $-1<y<1$; and digit $b:=+1$ in case $y>0$.
Then replace $y:=2(y-b)$ and $k:=2k+b$ and repeat.

Said loop gets executed $j$ times, where $j\in\IN$ 
has previously been determined similarly non-deterministically 
such that $y:=x\cdot 2^{-j}$ lies in $(-1;1)$; 
so $j\leq\calO(\log|x|)$.
As in Example~\ref{x:URound}, each
iteration of the loop body costs $\calO(\log|x|)$;
and due to their overlap, the three tests
wrapped in \texttt{choose()} also incur at most $\calO(\log|x|)$:
adding up to total cost $\calO(\log^2|x|)$.
\end{example}

\subsection{Real Turing Machine Simulation in \ERC}
\label{ss:Simul1}

Recall the oracle Turing machine model $\Machine^?$, defining polynomial time complexity
of partial real functions $f:\subseteq\IR^d\to\IR$ \cite[\S2.4+\S2.5]{Ko91}.
Essentially, on input of any precision parameter $p\in\IN$ and given access 
to any oracle tuple $\vec\Phi=(\Phi_1,\ldots,\Phi_d)$ encoding any $\vec x=(x_1,\ldots,x_d)\in\dom(f)$ in the below sense,
$\Machine^{\vec\Phi}(p)$ must terminate and output some integer $z_{p}$ such that $|f(\vec x)-z_{p}/2^{p}|\leq2^{-p}$
within a number of steps polynomial in $p+\myn(\vec x)$ regardless of $\vec\Phi$;
otherwise the behaviour of $\Machine^{\vec\Phi}(p)$ is arbitrary.
Here a (string-function) oracle $\Phi_j$ encoding
$x_j\in\IR$ is any total length-monotone mapping 
\begin{equation}
\label{e:Oracle}
\Phi_j:\ONE^*\to\TWO^{+} \quad\text{such that}\quad
|x_j-\bin\big(\Phi_j(\sdone^m)\big)/2^m|\leq2^{-m} \enspace , 
\end{equation}
where $\bin(b_0,\ldots,b_{n},b_{n+1})=(-1)^{b_{n+1}}\cdot(b_0+2b_1+\cdots+2^{n}b_{n})$
and $(b_{n-1},b_{n})\neq(0,0)$ in case $n\geq1$ to prevent excessive padding \cite[Example~7.2.3]{Wei00}.

\begin{theorem}
\label{t:Simul1}
Let $f(\vec x)\in\IR$ on domain $D\subseteq\IR^d$ 
be computed by an oracle Turing machine 
in polynomial time. 
There is an \ERC program realizing $f$ at polynomial cost.
\end{theorem}
Common simulations of (discrete) Turing machines by Register machines
\cite[Theorem~5]{Boas90} employ an unbounded number of registers with indirect addressing,
which we prohibit in \ERC, recall Subsection~\ref{ss:ERC}g).

\begin{proof}[Theorem~\ref{t:Simul1}]
W.l.o.g. let $\Gamma=\{0,\ldots,|\Gamma|-1\}$ denote the tape alphabet of $\Machine^?$, 
with $0$ for empty cells.
Similarly to \cite[Example~8.1]{BoolosBurgess}, encode
each tape's contents $(\gamma_0,\gamma_1,\ldots,\gamma_{h-1},\gamma_h,\gamma_{h+1},\ldots)$
as two natural numbers in $|\Gamma|$-ary expansion,
where $h$ denotes the current position of the read/write head:
$T:=\sum_{\eta\geq h} \gamma_{\eta}\cdot|\Gamma|^{\eta-h}$ and $T':=\sum_{\eta=0}^{h-1} \gamma_{\eta}\cdot|\Gamma|^{h-1-\eta}$.
Since $\Machine^?$ is promised to make at most $t(n)$ steps,
both $T$ and $T'$ have binary length bounded by $\calO(t(n))$.
Replacing $\gamma_h=0$ with something else amounts to one addition;
moving the head one cell left/right involves multiplying either $T$ or $T'$ by $|\Gamma|$,
via $|\Gamma|$-fold addition also at local cost $\calO(t(n))$ each
according to Definition~\ref{d:Contribution}.
Moreover the contents $\gamma_h$ of the cell at the current head position
can be extracted/erased using long division with remainder by $|\Gamma|$;
which boils down to $\calO(\polylog t(n))$ halving operations, 
each incurring bit-cost $\calO(t(n))$.

It remains to treat $\Machine$'s oracle queries and output.
For the former, replace $\Phi_j(\sdone^m)$ with $\Round(x_j\cdot2^m)$
at cost $\calO(m^2)$ according to Subsection~\ref{ss:Round}
since $x\in[0;1]$. Indeed, said query string $\sdone^m$ takes $m$ 
steps for $\Machine^?$ to compile, it follows $m\leq t(n)$.
And $\IN\ni m\mapsto 2^m\in\IR$ followed by real multiplication
with $x_j\in[0;1]$ costs $\calO(m)+\calO(m\cdot\log m)$ according 
to Definition~\ref{d:Contribution}; and $\Round()$ costs
$\calO(m^2)$ when implemented according to Example~\ref{x:BRound}
instead of Example~\ref{x:URound}.
This shows how to calculate within \ERC, given $n\in\IN$,
some $z_n\in\IZ$ with $|f(\vec x)-z/2^n|\leq2^{-n}$.
`Shadowing' the integer operations leading to said $z_n$
with real operations simultaneously yields $r_n\in\IR$
with $r_n=z_n$, at same asymptotic cost. 
And generating and multiplying with $2^{-n}\in\IR$
to obtain $y_n:=r_n/2^n\in\IR$ is also within budget.
Finally, \ERC's limit functional (Example~\ref{x:Limit})
`converts' the
real sequence $y_n$ of fast Cauchy approximations 
to $\lim_n y_n=f(\vec x)$, at cost 
$\poly\big(t(\calO(n))\big)$ 
according to Definition~\ref{d:Contribution}.
\qed\end{proof}

\section{Moduli of Continuity of \ERC Programs} 
\label{s:ERC2}

Definition~\ref{d:Contribution} had assigned local bit-costs
to the basic operations comprising \ERC:
depending on both length and working precision parameters $n$ and $m$,
whose values and evolution throughout a computation remains 
a (usually human) challenge to determine or bound;
recall Remark~\ref{r:Local}.
For the length parameter $n$, said evolution is based on a straight-forward induction
on the operations performed, starting with the program's arguments.
In \ERC the evolution of $m$ proceeds 
backwards: from desired output precision $p$
via intermediate working precisions
sufficient for real sign tests;
recall Subsection~\ref{ss:ERC}b+c).
Subsections~\ref{ss:Modulus2} and \ref{ss:Choose} above
have already introduced moduli of continuity for 
the primitives comprising \ERC; 
and the present section combines them 
to moduli of entire \ERC programs:
formalized first as Straight-Line Programs
(Subsection~\ref{ss:SLP} without branching,
then as Computation Trees (Subsection~\ref{ss:CT})
with branching, and finally as Computation Forests
(Subsection~\ref{ss:CF})
\ie collections of Computation Trees 
which may call each other 
including possible recursion,
and culminating in Subsection~\ref{ss:Simul2}
about the efficient simulation of \ERC on Turing machines.

\subsection{Straight-Line Programs}
\label{ss:SLP}

In order to formally assign moduli to real algorithms,
recall {\rm\cite[\S4.1]{ACT97}} the \emph{Straight-Line Program} model of computation,
here over a collection $\calF$ of (possibly partial) finitary operations $f$ on some many-sorted structure $\calS$.
Such a program $\calL$ (of arity $d$ and length $L$, say) produces, upon
input of $x_{-d},\ldots,x_{-1}\in\calS$, an $L$-element sequence $y_1,\ldots,y_L\in\calS$ as follows:
Each $y_\ell$ is the result of applying $f_\ell\in\calF$ 
to $d_\ell$ arguments from among previous results and inputs $(x_{-d},\ldots,x_{-1},y_1,\ldots,y_{\ell-1})$,
were $d_\ell$ denotes the arity of $f_\ell$. 
Formally write
$\pi_\ell:\{1,\ldots,d_\ell\}\to\{-d,\ldots,-1,+1,\ldots,\ell-1\}$ for
said `multiplexer' assignment reversed, 
namely from among the previous results 
(possibly repeated, aka \emph{fan-out}) 
to arguments of $f_\ell$, as mapping acting on tuple indices:
$y_\ell=f_\ell\circ\pi_\ell^{-1}(x_{-d},\ldots,x_{-1},y_1,\ldots,y_{\ell-1})$;
and $\calL=(f_1\circ\pi_1^{-1},\ldots,f_L\circ\pi_L^{-1})$ for the thus composed straight-line program (SLP).
Moreover abbreviate $\calL_{\ell}=(f_1\circ\pi_1^{-1},\ldots,f_\ell\circ\pi_\ell^{-1})$ for the initial segment of $\calL$,
where $\calL_{-j}$ means the SLP that merely re-produces input $x_{-j}$ ($1\leq j\leq d$).
With list concatenation notation 
$\calL=\calL_{L-1} \concat \big(f_L\circ\pi_L^{-1}\big)$,
the \emph{semantics} $\tilde\calL(\vec x)=(\vec x,\vec y)$ of $\calL$ on input $\vec x$
can be expressed inductively as 
\begin{equation}
\label{e:SLP}
\tilde\calL(\vec x) \;\;=\;\; 
\widetilde{\calL_{L-1}}(\vec x) \;\concat\:
\Big(f_{L}\circ\pi^{-1}_{L}\big(\widetilde{\calL_{L-1}}(\vec x)\big)\Big)
\enspace .
\end{equation}
Note that $\tilde\calL(\vec x)$ may contain undefined entries since the $f_\ell$ are generally partial:
We write $\dom(\calL)\subseteq\calS^d$ for the set of inputs $\vec x$ making each component $y_\ell$ of $\tilde\calL(\vec x)$ defined.

\begin{definition}[Modulus of a Straight-Line Program]
\label{d:SLP}
Let $\calL=(f_1\circ\pi_1^{-1},\ldots,f_L\circ\pi_L^{-1})$ be a SLP of arity $d$ on $\calS$ over $\calF$.
Suppose that $\calS$ is endowed with a metric 
and that $\calF$ consists of continuous operations only,
so that each $f\in\calF$ has a unique modulus $\mu_f$ of pointwise continuity 
according to Remark~\ref{r:Composition}.
Then assign to (the last result $y_L$ of) $\calL$ a \emph{pointwise modulus} 
$\mu_{\calL}:\dom(\calL)\times\uobarZ\nearrow\uobarZ^d$ by induction:
\begin{gather}
\mu_{\calL_{-j}}(\vec x,p) \;:=\; (-\infty,\ldots,-\infty,\overset{j\text{-th}}{p},-\infty,\ldots,-\infty), \quad j=1,\ldots,d 
\nonumber \\ \label{e:SLP2}
\mu_{(\calL_{L-1}\concat f_{L}\circ\pi^{-1}_{L})}(\vec x,p) \;:=\;
\max\nolimits_{k=1}^{d_L} \mu_{\calL_{\pi(k)}}\big(\vec x,\mu_{f_L,k}\big(\calL_{\pi(k)}(\vec x),p\big)\big)
\end{gather}
componentwise.
\end{definition}
Note that $\calL_{\pi(k)}$ may calculate, in addition
to $y_k$, intermediate results $y_\ell$ that $f_L\circ\pi_L^{-1}$
does not depend on: 
these `contribute' $-\infty$ to the maximum
in Equation~\eqref{e:SLP2};
recall Definition~\ref{d:Modulus}a).
We record that, by induction on Estimate~\eqref{e:Composition2},
$\mu_{\calL}$ is a modulus of local continuity of $\mu_{\tilde\calL}$.

\begin{example}
\label{x:SLP}
Even when $\mu_{\tilde\calL}$ and all $\mu_f$ are `optimal' moduli of continuity,
$\mu_{\calL}$ may be unnecessarily large by far 
and thus fail the second part of Problem~\ref{p:Stability2}:
Consider SLPs like $x-x+x-x+x-x\ldots$ and recall Remark~\ref{r:OverEstimate}.
\end{example}
Straight-line programs cannot branch, 
hence the subsequent generalization:

\subsection{Computation Trees}
\label{ss:CT}

Recall the concept of a \emph{Computation Tree} {\rm\cite[\S4.4]{ACT97}},
here over operations $\calF$ and (generalized, see Definition~\ref{d:MultiPred}) predicates $\calP$ on $\calS$:
\vspace{-\topsep}%
\begin{itemize}
\item
nodes of out-degree 1 called \emph{computation} nodes,
\item
nodes of out-degree $b\geq2$ called \emph{branching} nodes,
\item
leaves (out-degree 0) called \emph{result} nodes,
\item together with $d$ designated \emph{input} nodes 
$\imath_{-d},\ldots,\imath_{-1}$.
\end{itemize}
Each computation node and each result node $u$ in such a tree $\calT$ 
is labelled with some $d_u$-ary function $f_u\in\calF$, 
receiving arguments specified from among its predecessor computation or input nodes;
and each $b$-fold branching node $v$ in $\calT$ is labelled with some 
(possibly generalized $b$-fold) $d_v$-ary predicate $P_v\in\calP$, 
also receiving arguments specified from among its predecessor computation or input nodes:
see Figure~\ref{f:linct} for an example.
Intuitively an input $\vec x\in\calS^d$ 
starts by assigning values to $\imath_{-d},\ldots,\imath_{-1}$
and then proceeds through $\calT$, 
resulting in both computation and branching nodes being evaluated (if defined)
and, for the latter, continuing to the left/right/$\tau$-th subtree---depending 
on whether its generalized predicate label evaluates to \false/\true/$\tau$:
the `control flow' may be both partial and non-deterministic!

\begin{figure}[htb]\centerline{%
\includegraphics[width=0.8\textwidth]{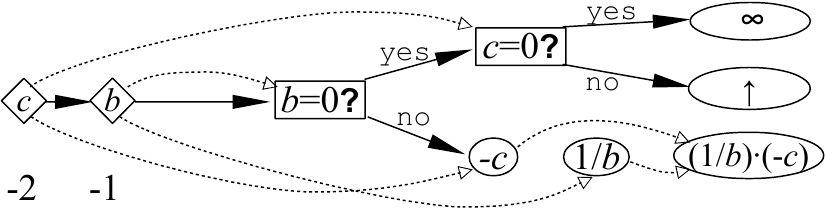}}
\caption{\label{f:linct}Toy example of a (total, binary) Computation Tree $\calT$
solving the linear equation $b\cdot x+c=0$: \newline
control flow is indicated with solid lines, 
information flow $\pi$ with dashed lines}
\end{figure}

Formally, each computation or branching or result node $w\in\calT$
is labelled with some $\omicron_w\in\calF\uplus\calP$
and additionally with some map $\pi_w:\{1,\ldots,d_w\}\to\calT'[w]$
indicating the inverse `information flow': namely
how arguments to $\omicron_w$ are fed with previous results.
Here $d_w\in\IN$ denotes the arity of $\omicron_w$, and
$\calT[w]$ denotes the set of input nodes in $\calT$ together with all of
$w$'s predecessor computation (but not branching) nodes 
including $w$ itself (even in case $w$ is a branching node);
same for $\calT'[w]=\calT[w]\setminus\{w\}$ but now excluding $w$ itself.
Analogous to SLPs, the semantics $\widetilde{\calT[u]}(\vec x)$ 
for a computation or result node $u$ on input $\vec x\in\calS^d$
(if defined, see next)
is the list of all `evaluated' computation and input (but not branching) nodes;
formally by induction:
\begin{equation}
\label{e:CT}
\widetilde{\calT[\imath_{-1}]}\;=\;\id:\calS^d\to\calS^d, \quad
\widetilde{\calT[u]} \;=\;\; \widetilde{\calT'[u]} \;\concat\; \Big(f_u\circ\pi_u^{-1}\big(\widetilde{\calT'[u]}\big)\Big)
\enspace .
\end{equation}
Similarly, 
$P_v\circ\pi_v^{-1}\big(\widetilde{\calT'[v]}\big)(\vec x)$
is the `value' of a branching node $v\in\calT$ on input $\vec x$,
provided that $\big(\widetilde{\calT'[v]}\big)(\vec x)$ is defined 
and lies in $\dom(P_v\circ\pi_v^{-1})$. Similarly, the domain
of a computation or result node $u$ according to Equation~\eqref{e:CT} is
\[ \dom(\calT[u]) \;=\;\; 
\dom(\calT'[u]) \;\cap\; 
\dom\big(f_u\circ\pi_u^{-1}(\calT'[u])\big) \enspace . \]
In case the predecessor of computation/result/branching node $w$ is a branching node $v$,
$\dom(\calT'[w])$ depends additionally on 
the `number' (=generalized truth value) $\tau$ of $w$ among $v$'s children:
\[ \dom(\calT'[w]) \;=\;\;
\big(\widetilde{\calT'[v]}\big)^{-1}\Big[ \pi_v\big[ P_v^{-1}[\tau]\big]\Big]
\enspace : \]
reflecting that only inputs making $P_v$ evaluate to $\tau$
in $v$ may proceed to $w$ in $\calT$,
possibly non-deterministically.

\begin{definition}[Modulus of a Computation Tree]
\label{d:CT}
Extending Definition~\ref{d:SLP},
suppose that all functions $f\in\calF$ and all generalized predicates $P\in\calP$ 
have associated moduli of pointwise continuity,
the latter in the sense of Definition~\ref{d:MultiPred}.
Let $\mu_w:\dom(\omicron_w)\times\uobarZ\nearrow\uobarZ^{d_w}$ 
denote said modulus associated with $\omicron_w\in\calF\uplus\calP$.
Define \emph{local} modulus 
$\mu_{\calT[w]}:\dom(\calT[w])\times\uobarZ\nearrow\uobarZ^d$ by induction,
analogous to Equations~\eqref{e:CT} and \eqref{e:SLP2}:
\begin{gather}
\mu_{\calT[\imath_{-j}]}(\vec x,p) \;:=\; (-\infty,\ldots,-\infty,\overset{j\text{-th}}{p},-\infty,\ldots,-\infty), \quad j=1,\ldots,d 
\nonumber \\ \label{e:CT2}
\mu_{\calT[w]}(\vec x,p) \;:=\;
\max\nolimits_{k=1}^{d_w} \mu_{\calT[\pi_w(k)]}
\Big(\vec x,\mu_{w,k}\big(\pi_w^{-1}\circ
\widetilde{\calT'[w]}
(\vec x)\:,\,p\big)\Big)
\enspace ,  \\ \nonumber 
\mu_{\calT}(\vec x,p) \;:=\; \max\big\{ \calT[u](\vec x,p) \::\:
u\text{ result node in $\calT$, }  \vec x\in\dom(\calT[u]) \big\} 
\end{gather}
componentwise. 
\end{definition}
Similarly to the case of SLPs, we record that
$\mu_{\calT[u]}$ is a pointwise modulus of continuity of $\widetilde{\calT[u]}$,
provided that $\mu_\omicron$ are moduli of pointwise continuity of $\omicron\in\calF\uplus\calP$.
Again, $\mu_{\calT[u]}$ may be arbitrarily larger than a least modulus of $\widetilde{\calT[u]}$,
if it exists; recall Example~\ref{x:SLP}.

Computation Trees do not natively support loops;
these can be `unrolled', if the number of iterations is known: 
compare Subsection~\ref{ss:Round}.
A popular exercise asks to
express Gaussian Elimination (Example~\ref{x:Determinant}) 
as computation Tree \cite[Example~4.23]{ACT97}.
Loops with indeterminate number of iterations 
can be expressed via recursive function calls.
Classical Computation Trees do not support 
calling other Computation Trees as `user functions', though:

\subsection{Computation Forests and \ERC}
\label{ss:CF}

A \emph{Computation Forest} $\frakF$ is a family of Computation Trees $\calT$
over the same set of basic operations $\calF$ and predicates $\calP$ on the joint structure $\calS$:
with the extension that a computation or result node in such $\calT\in\frakF$ may now, 
alternative to employing some operation $f\in\calF$, invoke some Computation Tree $\calT'\in\frakF$---possibly 
recursively and in particular including $\calT$ itself.
But differently, a computation forest $\frakF$ expands $\calF$ to $\calF\uplus\frakF$: syntactically.
The partial semantics $\widetilde{\calT[u]}(\vec x)$
of node $u$ in Computation Tree $\calT\in\frakF$ is again defined inductively by Equation~\eqref{e:CT},
and as least fixed point of the recursion operator \cite[\S6]{DenotationalSemantics};
similarly for the local modulus of continuity 
$\mu_{\calT}(\vec x,p)$ according to Equation~\eqref{e:CT2}.
Now tailor Computation Forests to the particular case of \ERC,
where cost contributes essentially in three ways:
from the number $N$ of basic operations/predicates executed,
from the magnitude $n$ of the argument/s to each such operation,
and from the working precision $m$ (=0 in the discrete case).

\begin{definition}
\label{d:ERC}
Let $\calS$ denote the following three-sorted structure:
Kleeneans $\IK=\{\false,\true,\unknown\}$; 
Presburger Integers $\big(\IZ,0_\IZ,1_\IZ,+_\IZ,-_\IZ,>,\div_2\big)$;
and the Reals as ordered field $\big(\IR,0,1,+,-,\times,\div,\sign\big)$;
with embedding $\error:\IZ\ni z\mapsto2^z\in\IR$, where 
the \emph{partial} Boolean real sign test $\sign:\IR\setminus\{0\}\to\{\true,\false\}$ 
doubles as total \emph{Kleenean} real sign test $\sign:\IR\to\{\true,\false,\unknown\}$.

Formally, the intra-sort operations comprising $\calF$ are 
\[ \{\true,\false,\vee,\wedge,\neg\} \;\uplus\; 
\{0_\IZ,1_\IZ,+_\IZ,-_\IZ,\div_2\} \;\uplus\;
\{0,1,+,-,\times,\div\}  \]
while $\error$ and $>$ and $\sign()$ and $\mychoose()$ 
connect different sorts. 
Here integer predicate $>$ and partial real $\sign()\big|_{\IR\setminus\{0\}}$ are collected in $\calP$; 
while the total Kleenean variant $\sign()$ is only permitted (in expressions
passed) as arguments to $\mychoose()$, thus capturing lazy evaluation.
\begin{enumerate}
\item[a)]
An \emph{\ERC Computation Forest} is syntactically a Computation Forest $\frakF$
over said $\calF$ and $\calP$ on $\calS$ such that semantically, for
each $\calT$ in $\frakF$ it holds: 
either all result nodes of $\calT$ are of sort $\IZ$ 
or all result nodes of $\calT$ are of sort $\IR$.
In the latter case, $\calT$ has $d+1\geq1$ input nodes,
and its last input node is of sort $\IZ$ 
and called \emph{output precision parameter} $p$;
moreover the return values $y=y(\vec x,p)$ of
$\calT$ on input $(\vec x,p)$,
given by $f_u\circ\pi^{-1}\big(\widetilde{\calT'[u]}\big)$
in result node $u$ according to Equation~\eqref{e:CT},
are supposed to satisfy 
$|y(\vec x,p)-y(\vec x,p')|\leq 2^{-p}+2^{-p'}$
whenever defined;
and `invoking' $\calT$ with the first $d$ inputs $\vec x$ 
but without said last input $p$
is understood to return
$\lim_{p\to\infty} y(\vec x,p)$.
\item[b\,i)]
Let $N=N_\calT(\vec x,p)\in\IN$
count the number of basic operations $f\in\calF$ 
and predicates $P\in\calP$
executed when invoking $\calT$ with input $(\vec x,p)$,
where such a call to $\calT'\in\frakF$ with parameters
$(\vec x',p')$ in turn adds $1+N_{\calT'}(\vec x',p')$
to $N_\calT(\vec x,p)$;
$N_\calT(\vec x,p)=\infty$ in case of infinite recursion,
or when $n_\calT(\vec x,p)=\infty$ or $m_\calT(\vec x,p)=\infty$
as defined next.
\item[b\,ii)]
Let $n=n_\calT(\vec x,p)\in\IN$
denote the largest magnitude $\myn(y)$
of any intermediate value $y$ arising 
during the computation of $\calT$ on input $(\vec x,p)$,
including when subsequently (and again possibly recursively) invoking of some $\calT'\in\frakF$;
$n_\calT(\vec x,p)=\infty$ in case of division by zero.
\item[b\,iii)]
Let $m=\mu_\calT(\vec x,p)\in\IN\cup\{\infty\}$ 
denote the initial/working precision 
according to Equation~\eqref{e:CT2},
sufficient for all real sign tests to succeed (including
when possibly recursively invoking some $\calT'\in\frakF$)
and for the result to have guaranteed error $\leq2^{-p}$;
$m:=\infty$ in case a real sign amounts to 0.
\item[c)]
We say that $\calT(\vec x,p)$ incurs \emph{polynomial cost}
on domain $D\subseteq\dom(\calT)$
provided that,
for all $p\in\IN$ and all $\vec x=(x_1,\ldots,x_d)\in D$,
all three $N_\calT(\vec x,p)$ and 
$n_\calT(\vec x,p)$ and $\mu_\calT(\vec x,p)$ according to (b)
are bounded by some polynomial in $p+\myn(\vec x)$.
\end{enumerate}
\end{definition}
Note that \ERC Computation Forests do not support---or need---neither variables nor loops:
the latter can be expressed by means of recursive calls among Computation Trees.
Subsection~\ref{ss:Round} above and Appendix~\ref{a:Steep} below shows and analyzes 
further example \ERC programs, including calls to others.
In the spirit of Remark~\ref{r:Local}c),
Definition~\ref{d:ERC}c) only distinguishes super/polynomial cost in \ERC:
future work may want to refine, for instance the degree of said polynomial cost,
see Question~\ref{q:Question} below. 

\subsection{Simulating \ERC on a Turing Machine}
\label{ss:Simul2}

Based on Definition~\ref{d:ERC}, we can now establish the converse to Theorem 24:

\begin{theorem}
\label{t:Simul2}
Suppose that $f:\subseteq\IR^d\to\IR$ on domain $\dom(f)\subseteq\IR^d$ is realized by
an ERC Computation Tree $\calT$, as part of forest $\frakF$, with polynomial cost. 
Then $f$ can be computed by an oracle Turing machine in polynomial time.
\end{theorem}
\begin{proof}[Theorem~\ref{t:Simul2}]
According to Definition~\ref{d:ERC}b\,iii+c), working precision $m$ 
polynomial in $n+p$ suffices to simulate all operations performed by $\calT(\vec x,p)$: 
including real sign tests and calls invoking some $\calT'\in\frakF$. 
And according to Definition~\ref{d:ERC}b\,ii+c), all intermediate
results have length $n'$ at most polynomial in $n+p$. 
Finally, according to Definition~\ref{d:ERC}b\,i+c),
only polynomially many basic operations get executed: 
each one incurring cost at most polynomial in $n+p$.
\qed\end{proof}
\noindent
Theorem~\ref{t:Simul1} basically says that our abstract ERC cost assignments are not unrealistically
large, compared to a Turing machine; and Theorem~\ref{t:Simul2} says that said cost assignments
are not too small, either. Computational experiments in Appendix~\ref{a:Experiments}
further confirm their agreement with practice. They in fact suggest that the polynomial
equivalence according to Theorems~\ref{t:Simul1} and \ref{t:Simul2} seems to hold on a finer level:

\begin{question}
\label{q:Question}
Can the generic polynomial relation established in Theorems~\ref{t:Simul1} and \ref{t:Simul2}
be tightened, such as to equivalence up to quadratic cost?
\end{question}

\subsection*{Acknowledgements}
\addcontentsline{toc}{subsection}{Acknowledgements}
Holger Thies is supported by JSPS KAKENHI Grant Numbers JP20K19744 and JP23K28036.
Martin Ziegler is supported by KAIST in-house grant KC30.
We thank anonymous reviewers.

\addcontentsline{toc}{section}{References}
\bibliographystyle{plain}
\bibliography{cca,erccost}

@article{BH98,
   author = {Brattka, Vasco and Hertling, Peter},
   title = {Feasible real random access machines},
   journal = {Journal of Complexity},
   volume = {14},
   number = {4},
   year = {1998},
   pages = {490--526},
}

@inproceedings{GL01,
   author = {Gowland, Paul and Lester, David},
   title = {A Survey of Exact Arithmetic Implementations},
   editor = {Blanck, Jens and Brattka, Vasco and Hertling, Peter},
   booktitle = {Computability and Complexity in Analysis},
   series = {Lecture Notes in Computer Science},
   volume = {2064},
   publisher = {Springer},
   address = {Berlin},
   year = {2001},
   pages = {30--47},
   note = {4th International Workshop, CCA 2000, Swansea, UK, September 2000},
}

@article{Grz57,
   author = {Grzegorczyk, Andrzej},
   title = {On the definitions of computable real continuous functions},
   journal = {Fundamenta Mathematicae},
   volume = {44},
   year = {1957},
   pages = {61--71},
}

@book{Ko91,
   author = {Ko, Ker-I},
   title = {Complexity Theory of Real Functions},
   series = {Progress in Theoretical Computer Science},
   publisher = {Birkh\"auser},
   address = {Boston},
   year = {1991},
}

@book{Koh08a,
   author = {Kohlenbach, Ulrich},
   title = {Applied Proof Theory: Proof Interpretations and their Use in Mathematics},
   publisher = {Springer},
   address = {Berlin},
   year = {2008},
}

@article{Lam07,
   author = {Lambov, Branimir},
   title = {{RealLib}: {A}n efficient implementation of exact real arithmetic},
   journal = {Mathematical Structures in Computer Science},
   volume = {17},
   year = {2007},
   pages = {81--98},
}

@article{Luc77,
   author = {Luckhardt, Horst},
   title = {A fundamental effect in computations on real numbers},
   journal = {Theoretical Computer Science},
   volume = {5},
   number = {3},
   year = {1977},
   pages = {321--324},
}

@article{Men05,
   author = {M{\'e}nissier-Morain, Val{\'e}rie},
   title = {Arbitrary precision real arithmetic: design and algorithms},
   journal = {The Journal of Logic and Algebraic Programming},
   volume = {64},
   year = {2005},
   pages = {13--39},
}

@inproceedings{Mue01,
   author = {M{\"u}ller, Norbert Th.},
   title = {The {iRRAM}: Exact Arithmetic in {C}++},
   editor = {Blanck, Jens and Brattka, Vasco and Hertling, Peter},
   booktitle = {Computability and Complexity in Analysis},
   series = {Lecture Notes in Computer Science},
   volume = {2064},
   publisher = {Springer},
   address = {Berlin},
   year = {2001},
   pages = {222--252},
   note = {4th International Workshop, CCA 2000, Swansea, UK, September 2000},
}

@inproceedings{Mue86a,
   author = {M{\"u}ller, Norbert Th.},
   title = {Subpolynomial complexity classes of real functions and real numbers},
   editor = {Kott, Laurent},
   booktitle = {Proceedings of the 13th International Colloquium on Automata, Languages, and Programming},
   series = {Lecture Notes in Computer Science},
   volume = {226},
   publisher = {Springer},
   address = {Berlin},
   year = {1986},
   pages = {284--293},
}

@article{PZ13,
   author = {Pauly, Arno and Ziegler, Martin},
   title = {Relative computability and uniform continuity of relations},
   journal = {Journal of Logic and Analysis},
   volume = {5},
   number = {7},
   year = {2013},
   pages = {1--39},
}

@book{Wei00,
   author = {Weihrauch, Klaus},
   title = {Computable Analysis},
   publisher = {Springer},
   address = {Berlin},
   year = {2000},
}

@article{Wei03,
   author = {Weihrauch, Klaus},
   title = {Computational complexity on computable metric spaces},
   journal = {Mathematical Logic Quarterly},
   volume = {49},
   number = {1},
   year = {2003},
   pages = {3--21},
}

@misc{NuclearGandhi,
  title = "Nuclear {G}handi",
  author = "{W}ikipedia{,} The Free Encyclopedia",
  year = "accessed August 2026",
  note = "http://en.wikipedia.org/wiki/Nuclear\_Gandhi",
}

@book{Davis1965,
	address = {Hewlett, NY, USA},
	editor = {Martin Davis},
	publisher = {Dover Publication},
	title = {The Undecidable: Basic Papers on Undecidable Propositions, Unsolvable Problems and Computable Functions},
	year = {1965}
}

@book{MCA,
place={Cambridge}, 
edition={3}, 
title={Modern Computer Algebra}, 
publisher={Cambridge University Press}, 
author={von zur Gathen, Joachim and Gerhard, Jürgen}, 
year={2013}}

@article{IntegerMultiplication,
author = {David Harvey and Joris van der Hoeven},
title = {Integer multiplication in time {O(n log n)}}, 
pages = {563--617},
volume = {193},
year = {2021},
journal = {Annals of Mathematics},
publisher = {Princeton},
}

@book{Dietzfelbinger,
author = {Dietzfelbinger, Martin},
year = {2004},
title = {Primality testing in polynomial time},
series = {LNCS},
volume = {3000},
publisher = {Springer},
isbn = {3-540-40344-2},
}

@techreport{KaratsubaSqrt,
author = {Paul Zimmermann},
title = {Karatsuba Square Root},
year = {1999},
url = {https://inria.hal.science/inria-00072854},
TYPE = {Research Report},
  NUMBER = {RR-3805},
  PAGES = {8},
  INSTITUTION = {{INRIA}},
}

@book{MullerElementary,
author       = {Jean{-}Michel Muller},
  title        = {Elementary functions: Algorithms and implementation},
  publisher    = {Birkh{\"{a}}user},
  year         = {1997},
  isbn         = {081763990X},
}

@Misc{Lester01,
author = {David Lester},
title = {Infinite Precision Basic Linear Algebra Systems or {Why Norbert M\"{u}ller is Right}},
year = {2001},
month = {November},
note = {http://cca-net.de/cca2001},
howpublished = {Dagstuhl 01461},
}

@article{Borwein2,
author = {Borwein, J. M. and Borwein, P. B.},
title = {On the Complexity of Familiar Functions and Numbers},
journal = {SIAM Review},
volume = {30},
number = {4},
pages = {589-601},
year = {1988},
doi = {10.1137/1030134},
}

@article{MPFR,
author = {Fousse, Laurent and Hanrot, Guillaume and Lef\`{e}vre, Vincent and P\'{e}lissier, Patrick and Zimmermann, Paul},
title = {{MPFR}: A multiple-precision binary floating-point library with correct rounding},
year = {2007},
issue_date = {June 2007},
publisher = {Association for Computing Machinery},
address = {New York, NY, USA},
volume = {33},
number = {2},
issn = {0098-3500},
doi = {10.1145/1236463.1236468},
journal = {ACM Trans. Math. Softw.},
month = jun,
pages = {13–es},
numpages = {15},
}

@article{ARIADNE,
title = {Rigorous Function Calculi in {A}riadne},
author = {Pieter Collins and Luca Geretti and Sanja Zivanovic Gonzalez and Davide Bresolin and Tiziano Villa},
doi = {10.48550/arXiv.2306.17541},
journal = {Logical Methods in Computer Science}, 
volume = {21}, 
issue = {3},
year = {2025},
}

@article{PBC+24b,
   author = {Park, Sewon and Brau{\ss}e, Franz and Collins, Pieter and Kim, SunYoung and Kone{\v{c}}n{\'{y}}, Michal and Lee, Gyesik and M{\"u}ller, Norbert and Neumann, Eike and Preining, Norbert and Ziegler, Martin},
   title = {Semantics, Specification Logic, and {H}oare Logic of Exact Real Computation},
   journal = {LMCS},
   volume = {20},
   number = {2},
   year = {2024},
   month = Jun,
   doi = {10.46298/lmcs-20(2:17)2024},
}

@book{MATLAB,
year = {2025},
author = {MATLAB},
title = {},
publisher = {The MathWorks Inc.},
address = {Natick, Massachusetts, United States},
url = {https://www.mathworks.com}
}

@book{TPbook,
author = {Arnold Sch\"{o}nhage and Andreas F. Grotefeld and Ekkehart Vetter},
title = {Fast Algorithms. A Multitape Turing Machine Implementation},
isbn = {3411168919},
publisher = {BI Wissenschaftsverlag},
year = {1994},
url = {https://pages.iai.uni-bonn.de/schoenhage_arnold/tp/TPbook.html},
}

@article{Lago22,
  author       = {Ugo Dal Lago},
  title        = {Implicit computation complexity in higher-order programming languages: {A} Survey in Memory of Martin Hofmann},
  journal      = {Math. Struct. Comput. Sci.},
  volume       = {32},
  number       = {6},
  pages        = {760--776},
  year         = {2022},
  doi          = {10.1017/S0960129521000505},
}

@book{Minsky67,
author = {Minsky, Marvin L.},
title = {Computation: finite and infinite machines},
year = {1967},
isbn = {0131655639},
publisher = {Prentice-Hall},
address = {USA},
}

@book{Hypercomputation,
author = {Syropoulos, Apostolos},
title = {Hypercomputation: Computing Beyond the Church-Turing Barrier (Monographs in Computer Science)},
year = {2007},
isbn = {0387308865},
publisher = {Springer},
address = {Berlin, Heidelberg}
}

@book{Automata,
author = {Hopcroft, John E. and Motwani, Rajeev and Ullman, Jeffrey D.},
title = {Introduction to Automata Theory, Languages, and Computation (3rd Edition)},
year = {2006},
isbn = {0321455363},
publisher = {Addison-Wesley Longman Publishing Co., Inc.},
address = {USA}
}

@book{BCSS98,
  title={Complexity and Real Computation},
  author={Lenore Blum and Felipe Cucker and Michael Shub and Steve Smale},
  year={1998},
  publisher={Springer},
  isbn={978-0-387-98281-6}
}

@article{Blum67,
  author       = {Manuel Blum},
  title        = {A Machine-Independent Theory of the Complexity of Recursive Functions},
  journal      = {J. {ACM}},
  volume       = {14},
  number       = {2},
  pages        = {322--336},
  year         = {1967},
  doi          = {10.1145/321386.321395},
}

@incollection{Boas90,
  author       = {Peter van Emde Boas},
  editor       = {Jan van Leeuwen},
  title        = {Machine Models and Simulation},
  booktitle    = {Handbook of Theoretical Computer Science, Volume {A}},
  chapter      = {Algorithms and Complexity},
  pages        = {1--66},
  publisher    = {Elsevier and {MIT} Press},
  year         = {1990},
}

@article{Yao03,
  author       = {Andrew Chi{-}Chih Yao},
  title        = {Classical physics and the Church-Turing Thesis},
  journal      = {J. {ACM}},
  volume       = {50},
  number       = {1},
  pages        = {100--105},
  year         = {2003},
  doi          = {10.1145/602382.602411},
}

@article{Deutsch85,
    author = {Deutsch, David},
    title = {Quantum theory, the Church–Turing principle and the universal quantum computer},
    journal = {Proceedings of the Royal Society of London. A. Mathematical and Physical Sciences},
    volume = {400},
    number = {1818},
    pages = {97-117},
    year = {1985},
    month = {07},
    issn = {0080-4630},
    doi = {10.1098/rspa.1985.0070},
}

@article{Ziegler09,
  author       = {Martin Ziegler},
  title        = {Physically-relativized Church-Turing Hypotheses: Physical foundations of computing and complexity theory of computational physics},
  journal      = {Appl. Math. Comput.},
  volume       = {215},
  number       = {4},
  pages        = {1431--1447},
  year         = {2009},
  doi          = {10.1016/J.AMC.2009.04.062},
}

@article{BeggsTucker14,
author = {Edwin J. Beggs and Jos\'{e} F\'{e}lix Costa and Diogo Po\c{c}as and John V. Tucker},
title = {An Analogue-Digital {Church}-{Turing} Thesis},
journal = {International Journal of Foundations of Computer Science},
volume = {25},
number = {04},
pages = {373-389},
year = {2014},
doi = {10.1142/S0129054114400012},
}

@book{ComputationalGeometry,
  author       = {Mark de Berg and
                  Otfried Cheong and
                  Marc J. van Kreveld and
                  Mark H. Overmars},
  title        = {Computational geometry: algorithms and applications, 3rd Edition},
  publisher    = {Springer},
  year         = {2008},
  doi          = {10.1007/978-3-540-77974-2},
  isbn         = {9783540779735},
}

@inproceedings{Sewon23,
  author       = {Sewon Park},
  editor       = {Henning Fernau and
                  Klaus Jansen},
  title        = {Verified Exact Real Computation with Nondeterministic Functions and
                  Limits},
  booktitle    = {Fundamentals of Computation Theory - 24th International Symposium,
                  {FCT} 2023, Trier, Germany, September 18-21, 2023, Proceedings},
  series       = {Lecture Notes in Computer Science},
  volume       = {14292},
  pages        = {363--377},
  publisher    = {Springer},
  year         = {2023},
  doi          = {10.1007/978-3-031-43587-4\_26},
}

@article{Sewon25,
  author       = {Michal Kone{\v{c}}n{\'{y}} and Sewon Park and Holger Thies},
  title        = {Extracting efficient exact real number computation from proofs in constructive type theory},
  journal      = {J. Log. Comput.},
  volume       = {35},
  number       = {6},
  year         = {2025},
  doi          = {10.1093/LOGCOM/EXAE066},
}

@article{Clerical,
  author       = {Andrej Bauer and Sewon Park and Alex Simpson},
  title        = {An Imperative Language for Verified Exact Real-Number Computation},
  journal      = {CoRR},
  volume       = {abs/2409.11946},
  year         = {2024},
  doi          = {10.48550/ARXIV.2409.11946},
  eprinttype    = {arXiv},
  eprint       = {2409.11946},
}

@article{BournezGracaPouly17,
  author       = {Olivier Bournez and Daniel Silva Gra{\c{c}}a and Amaury Pouly},
  title        = {Polynomial Time Corresponds to Solutions of Polynomial Ordinary Differential Equations of Polynomial Length},
  journal      = {J. {ACM}},
  volume       = {64},
  number       = {6},
  pages        = {38:1--38:76},
  year         = {2017},
  doi          = {10.1145/3127496},
}

@article{LimZiegler25,
  author = {Donghyun Lim and Martin Ziegler},
  title = {Quantitative Coding and Complexity Theory of Continuous Data},
  journal      = {J. {ACM}},
  volume       = {72},
  number       = {1},
  pages        = {4:1--4:39},
  year         = {2025},
  doi          = {10.1145/3705609},
}

@book{Condition,
author = {B\"{u}rgisser, Peter and Cucker, Felipe},
title = {Condition: The Geometry of Numerical Algorithms},
year = {2013},
isbn = {3642388957},
publisher = {Springer},
}

@InProceedings{Aras25,
author = {Aras Bacho and Martin Ziegler},
editor       = {Fran{\c{c}}ois Boulier and Chenqi Mou and Timur M. Sadykov and Evgenii V. Vorozhtsov},
  title        = {Second-Order Parameterizations for the Complexity Theory of Integrable Functions},
  booktitle    = {Proc. 27th {CASC}}, 
  series       = {LNCS},
  volume       = {16235},
  pages        = {27--46},
  publisher    = {Springer},
  year         = {2025},
  doi          = {10.1007/978-3-032-09645-6\_2},
}

@article{MeerZiegler08,
  author       = {Klaus Meer and
                  Martin Ziegler},
  title        = {An explicit solution to {P}ost's Problem over the reals},
  journal      = {J. Complex.},
  volume       = {24},
  number       = {1},
  pages        = {3--15},
  year         = {2008},
  doi          = {10.1016/J.JCO.2006.09.004},
}

@inproceedings{FournierKoiran98,
  author       = {Herv{\'{e}} Fournier and
                  Pascal Koiran},
  editor       = {Jeffrey Scott Vitter},
  title        = {Are Lower Bounds Easier over the Reals?},
  booktitle    = {Proceedings of the Thirtieth Annual {ACM} Symposium on the Theory
                  of Computing, Dallas, Texas, USA, May 23-26, 1998},
  pages        = {507--513},
  publisher    = {{ACM}},
  year         = {1998},
  doi          = {10.1145/276698.276864},
}

@book{ACT97,
title = {Algebraic Complexity Theory},
author = {Peter Bürgisser, Michael Clausen, Mohammad Amin Shokrollahi},
series = {Grundlehren der mathematischen Wissenschaften},
doi = {10.1007/978-3-662-03338-8},
publisher = {Springer},
year = {1997},
}

@article{ZieglerKoolen08,
title = {Kolmogorov Complexity Theory over the Reals},
journal = {Electronic Notes in Theoretical Computer Science},
volume = {221},
pages = {153-169},
year = {2008},
issn = {1571-0661},
doi = {https://doi.org/10.1016/j.entcs.2008.12.014},
author = {Martin Ziegler and Wouter M. Koolen},
}

@article{SchaeferCardinalMiltzow24,
  author       = {Marcus Schaefer and
                  Jean Cardinal and
                  Tillmann Miltzow},
  title        = {The Existential Theory of the Reals as a Complexity Class: {A} Compendium},
  journal      = {CoRR},
  volume       = {abs/2407.18006},
  year         = {2024},
  doi          = {10.48550/ARXIV.2407.18006},
  eprinttype    = {arXiv},
  eprint       = {2407.18006},
}

@article{Rice54,
author = {Henry Gordon Rice},
title = {Recursive Real Numbers},
journal = {Proc. {AMS}},
 number = {5},
 pages = {784--791},
 publisher = {American Mathematical Society},
 volume = {5},
 year = {1954},
 doi = {https://doi.org/10.2307/2031867},
}

@book{Kleene52,
author = {Stephen Cole Kleene},
  publisher = {D. van Nostrand},
  title = {Introduction to Metamathematics},
  year = 1952,
  isbn = {1258437961},
}

@inproceedings{Soft,
  author       = {Chee Yap and Michael Sagraloff and Vikram Sharma},
  editor       = {Paola Bonizzoni and Vasco Brattka and Benedikt L{\"{o}}we},
  title        = {Analytic Root Clustering: {A} Complete Algorithm Using Soft Zero Tests},
  booktitle    = {Proc. 9th Conf. Computability in Europe},
  series       = {LNCS},
  pages        = {434--444},
  publisher    = {Springer},
  year         = {2013},
  doi          = {10.1007/978-3-642-39053-1\_51},
}

@inproceedings{core2,
  author       = {Jihun Yu and
                  Chee Yap and
                  Zilin Du and
                  Sylvain Pion and
                  Herv{\'{e}} Br{\"{o}}nnimann},
  editor       = {Komei Fukuda and
                  Joris van der Hoeven and
                  Michael Joswig and
                  Nobuki Takayama},
  title        = {The Design of Core 2: {A} Library for Exact Numeric Computation in Geometry and Algebra},
  booktitle    = {Proc. 3rd {ICMS} International Congress on Mathematical Software},
  series       = {LNCS},
  pages        = {121--141},
  publisher    = {Springer},
  year         = {2010},
  doi          = {10.1007/978-3-642-15582-6\_24},
}

@article{realLib,
   author = {Lambov, Branimir},
   title = {{RealLib}: {A}n efficient implementation of exact real arithmetic},
   journal = {MSCS},
   volume = {17},
   year = {2007},
   pages = {81--98},
}

@InProceedings{TaylorModels,
  author =	{Park, Sewon and Thies, Holger},
  title =	{{A Coq Formalization of Taylor Models and Power Series for Solving Ordinary Differential Equations}},
  booktitle =	{15th International Conference on Interactive Theorem Proving (ITP)},
  pages =	{30:1--30:19},
  series =	{{LIPIcs}},
  ISBN =	{978-3-95977-337-9},
  ISSN =	{1868-8969},
  year =	{2024},
  volume =	{309},
  editor =	{Bertot, Yves and Kutsia, Temur and Norrish, Michael},
  publisher =	{Schloss Dagstuhl Leibniz Center for Informatics},
  doi =		{10.4230/LIPIcs.ITP.2024.30},
}

@article{CookReckhow73,
  author       = {Stephen A. Cook and
                  Robert A. Reckhow},
  title        = {Time Bounded Random Access Machines},
  journal      = {J. Comput. Syst. Sci.},
  volume       = {7},
  number       = {4},
  pages        = {354--375},
  year         = {1973},
  doi          = {10.1016/S0022-0000(73)80029-7},
}

@book{Papadimitriou,
  author = {Papadimitriou, Christos H.},
  isbn = {978-0-201-53082-7},
  publisher = {Addison-Wesley},
  title = {Computational complexity},
  year = 2005
}

@article{LambdaCost,
author = {Lawall, Julia L. and Mairson, Harry G.},
title = {Optimality and inefficiency: what isn't a cost model of the lambda calculus?},
year = {1996},
issue_date = {June 15, 1996},
publisher = {Association for Computing Machinery},
address = {New York, NY, USA},
volume = {31},
number = {6},
issn = {0362-1340},
doi = {10.1145/232629.232639},
journal = {SIGPLAN Not.},
month = jun,
pages = {92–101},
numpages = {10}
}

@article{LambdaMachine,
title = {The weak lambda calculus as a reasonable machine},
journal = {Theoretical Computer Science},
volume = {398},
number = {1},
pages = {32-50},
year = {2008},
issn = {0304-3975},
doi = {https://doi.org/10.1016/j.tcs.2008.01.044},
author = {Ugo {Dal Lago} and Simone Martini},
}

@inproceedings{MullerZ14,
  author       = {Norbert Th. M{\"{u}}ller and
                  Martin Ziegler},
  editor       = {Hoon Hong and
                  Chee Yap},
  title        = {From Calculus to Algorithms without Errors},
  booktitle    = {Mathematical Software - {ICMS} 2014 - 4th International Congress,
                  Seoul, South Korea, August 5-9, 2014. Proceedings},
  series       = {Lecture Notes in Computer Science},
  pages        = {718--724},
  publisher    = {Springer},
  year         = {2014},
  doi          = {10.1007/978-3-662-44199-2\_107},
}

@article{AllenderBKM09,
  author       = {Eric Allender and Peter B{\"{u}}rgisser and Johan Kjeldgaard{-}Pedersen and Peter Bro Miltersen},
  title        = {On the Complexity of Numerical Analysis},
  journal      = {{SIAM} J. Comput.},
  volume       = {38},
  number       = {5},
  pages        = {1987--2006},
  year         = {2009},
  doi          = {10.1137/070697926},
}

@book{SICP,
author = {Abelson, Harold and Sussman, Gerald J. and Sussman, Julie},
publisher = {MIT Press},
title = {{Structure and Interpretation of Computer Programs}},
year = {2010},
edition = {Second},
isbn = {0262011530},
}

@misc{Aern,
  author       = {Michal Kone{\v{c}}n{\'{y}} and Sewon Park and Holger Thies},
  title        = {{cAERN} library},
  publisher    = {{DROPS} Artifacts},
  year         = {2024},
  month        = Nov,
  doi          = {10.4230/ARTIFACTS.22444},
  note         = {doi:10.4230/ARTIFACTS.22444},
}

@article{VANDERHOEVEN200652,
title = {Computations with effective real numbers},
journal = {Theoretical Computer Science},
volume = {351},
number = {1},
pages = {52-60},
year = {2006},
note = {Real Numbers and Computers},
issn = {0304-3975},
doi = {10.1016/j.tcs.2005.09.060},
author = {Joris {van der Hoeven}},
}

@InCollection{Neumaier1993,
author="Neumaier, A.",
editor="Albrecht, R.
and Alefeld, G.
and Stetter, H. J.",
title="The Wrapping Effect, Ellipsoid Arithmetic, Stability and Confidence Regions",
bookTitle="Validation Numerics: Theory and Applications",
year="1993",
publisher="Springer Vienna",
address="Vienna",
pages="175--190",
isbn="978-3-7091-6918-6",
doi="10.1007/978-3-7091-6918-6_14"
}

@book{DenotationalSemantics,
author = {Schmidt, David A.},
year = {1986},
title = {Denotational semantics: a methodology for language development},
publisher = {Allyn \& Bacon}
}

@inproceedings{Emperor,
   author = {Brattka, Vasco},
   title = {The Emperor's New Recursiveness: The Epigraph of the Exponential Function in Two Models of Computability},
   editor = {Ito, Masami and Imaoka, Teruo},
   booktitle = {Words, Languages \& Combinatorics III},
   publisher = {World Scientific Publishing},
   address = {Singapore},
   year = {2003},
   pages = {63--72},
}

@inproceedings{BlancBournez22,
  author       = {Manon Blanc and Olivier Bournez},
  editor       = {J{\'{e}}r{\^{o}}me Durand{-}Lose and Gy{\"{o}}rgy Vaszil},
  title        = {A Characterization of Polynomial Time Computable Functions from the
                  Integers to the Reals Using Discrete Ordinary Differential Equations},
  booktitle    = {Proc. 9th International Conference on Machines, Computations, and Universality {MCU}},
  series       = {LNCS},
  volume       = {13419},
  pages        = {58--74},
  publisher    = {Springer},
  year         = {2022},
  doi          = {10.1007/978-3-031-13502-6\_4},
}

@book{BoolosBurgess,
  author    = {Boolos, George S. and Burgess, John P. and Jeffrey, Richard C.},
  title     = {Computability and Logic},
  publisher = {Cambridge University Press},
  year      = {1974},
  address   = {Cambridge},
  isbn      = {052120402X}
}

@InProceedings{MCU2026,
 author = {Jihoon Hyun and Holger Thies and Martin Ziegler},
 title = {Towards Algorithmic Cost in Exact Real Computation},
 booktitle = {Proc. 11th Conference on Machines, Computations and Universality ({MCU})},
 publisher = "Springer",
 year = "2026",
 editor = "Henning Fernau and Serghei Verlan",
}

\appendix
\section{\ERC Example Program and Analysis of \emph{Steep} Function $1/\ln(e/x)$}
\label{a:Steep}

Recall the function $h$ from Remark~\ref{r:Modulus}d),
with continuous extension to $x=0$.
The following \ERC program implements said $h$:
    \begin{algorithmic}[0] 
    \Function{Steep}{$x:[0;1];\; p:\IN$}
    \State \textbf{var} $r:\IR=0$;\; $e:\IR=\exp(1)$;\; $x0:\IR=1/e$;
    \Comment{Calculate $x0:=\exp(-2^p)$}
    \State \textbf{var} $j:\IZ$;
    \For{$j=1\ldots p$}
       \State{$x0=x0\cdot x0$;}
    \EndFor
    \Comment{Avoid division by $x=0$}
    \If{$\mychoose(x<e\cdot x0\:,\:x>x0)==1$}
	$r=1/\ln(e/x)$; \EndIf
    \State \Return $r$;
    \EndFunction
    \end{algorithmic}
Note the separate types for reals and integers,
the use of $\mychoose()$ with `overlapping' conditions,
as well as the \emph{exact} calls to the exponential function from \cite[\S5.2]{PBC+24b}
and to the natural logarithm function below,
although their return statements need and do provide merely approximations 
up to error $2^{-p}$ due to the implicit limit; recall Subsection~\ref{ss:ERC}g).
To compute $\big(h\circ h\big)(x)$ in \ERC now, simply call $\textsc{Steep}\big(\textsc{Steep}(x)\big)$.
    \begin{algorithmic}[0] 
    \Function{ln}{$x:\IR;\; p:\IN$}  \Comment $x>0$
    \State \textbf{var} $j:\IZ=1$; \; \textbf{var} $jr:\IR=1$;
    \State \textbf{var} $r:\IR=0$;\; $e:\IR=\exp(1/2)$;
    \Comment Range reduction to \; $\tfrac{1}{e}\leq x\leq e$:
    \While{$\mychoose(x<3/2\:,\:x>1)==1$}
       \State $r=r+1/2$; \; $x=x/e$;
    \EndWhile
    \While{$\mychoose(x>\tfrac{1}{2}\:,\:x<1)==1$}
       \State $r=r-\tfrac{1}{2}$; \; $x=x\cdot e$;
    \EndWhile
    \Comment $|x-1|\leq \sqrt{1/2}$: Taylor series, $2p$ terms
    \State $e=x-1$; \; $x=e$; 
    \For{$j=1\ldots p+p$}
       \State $r=r+x/jr$; \; $x=-x\cdot e$; \; $jr=jr+1$;
    \Comment $j\equiv jr$, but different types
    \EndFor
    \State \Return $r$;
    \EndFunction
    \end{algorithmic}
Regarding the computational cost incurred by $\textsc{ln}(x,p)$,
the range reduction while-loops repeat $y+\calO(1)$ times where $y:=\max\{x,1/x\}$;
and each repetition has bit-cost $\calO\big((m+\log y)\cdot\polylog(m+\log y)\big)$,
where $m$ denotes the working precision; recall Figure~\ref{f:FPCost}.
Since each iteration `looses' $\calO(1)$ bits of precision, $m=\calO(p+y)$. 
The final for-loop is repeated $2p$ times,
each costing $\calO(p\cdot\polylog p)$:
in total $\calO\big((y+p)^2\cdot\polylog(y+p)\big)$.

Similarly, $\textsc{exp}(x,p)$ in \cite[\S5.2]{PBC+24b}
consists of a range reduction loop repeated $\calO(|x|)$ times
of bit-cost $\calO\big((m+\log|x|)\cdot\polylog(m+\log|x|)\big)$ each,
$m=\calO(p+|x|)$; followed by the Taylor expansion's first $p$ terms,
each of cost $\calO(p\cdot\polylog p)$:
in total $\calO\big((|x|+p)^2\cdot\polylog(|x|+p)\big)$.

\begin{remark}
\label{r:Steep}
The above algorithms for logarithm and exponential function 
are not optimal: quadratic in the output precision $p$
and even exponential in the length of the real argument $x$.
Indeed our goal here is to illustrate \ERC example programs
and their cost analyses.
Also, in spite of calling said sub-optimal \textsc{exp}
and \textsc{ln}, \textsc{Steep} above is still close to
optimal according to Remark~\ref{r:Modulus}d+e),
as analyzed next.
Subsection~\ref{aa:Steep} shows runtime measurements,
of approximating $\textsc{Steep}(x)$ 
for various choices of $x\ll1$
up to error $2^{-p}$ for $p=50$:
the minimum (not maximum) precision supported by 
the \textsf{iRRAM} library.
\end{remark}
Regarding $\textsc{Steep}(x,p)$,
$x0$ gets initialized to $1/\exp(1)$:
at cost $\calO(m^2\cdot\polylog m)$, as analyzed above,
in the working precision $m$.
And the for-loop then repeatedly squares $x0<1$
in order to obtain $x0=\exp(-2^p)$,
each iteration incurring $\calO(m\cdot\polylog m)$
and loosing $\calO(1)$ bits of precision.
For the $\mychoose()$ command to succeed, 
the working precision must be sufficient 
to reliably determine at least one of the 
two $\sign(x-x0)$ and $\sign(x-e\cdot x0)$.
In case $x>x0=\exp(-2^p)$, this requires $m>2^p/\ln2$;
leads to calculating $\ln(y)$ for $y:=e/x$, 
by calling $\textsc{ln}(y)$ as analyzed above at cost 
$\calO\big((p+y)^2\cdot\polylog(p+y)\big)
\leq\calO\big((p+1/x)^2\cdot\polylog(p+1/x)\big)
\leq2^{-\calO(p)}$.
Verifying the inequality $x<e\cdot x0=\exp(1-2^p)$ 
on the other hand takes working precision $m>(1-2^p)/\ln2$.
Either way the global cost is $2^{\calO(p)}$.

\section{Computational Experiments}
\label{a:Experiments}

The following implementations and measurements were performed with the
\textsf{iRRAM} \texttt{C++} library (version 2014\_01, backend \textsf{MPFR})
running under \textsf{Ubuntu}
on a (virtual single core of an) Intel Broadwell CPU at 2.3GHz with 4GB RAM.

\subsection{Integer Rounding Multifunction}
\label{aa:Round}

Running times
of the `unary' (top) 
and the `binary' (bottom) multivalued integer rounding algorithms
from Examples~\ref{x:URound} and \ref{x:BRound}:
first with linear fit $-45.3+0.00058|x|$,
second with logarithmic fit $-22.7+0.058\lg|x|$.

\begin{tikzpicture}
\begin{axis}
[
    xlabel={$x$}, 
    ylabel={runtime/msec},
    width=0.80\textwidth,
    height=0.4\textwidth,
    scale only axis,
    ymin=10,
    xmax=19000000,
    xmin=-19000000,
    yticklabel style={/pgf/number format/fixed},
    scaled y ticks=false,
    scaled x ticks=false
    ]
\addplot[only marks] table[col sep=comma]{unaryround.csv};
\addplot[domain=-10000000:10000000]{-45.3385 + 0.0005845*abs(x)};
\end{axis}
\end{tikzpicture}

\medskip\noindent
\begin{tikzpicture}
\begin{axis}
[
    xlabel={$\lg|x|$}, 
    ylabel={runtime/msec},
    width=0.80\textwidth,
    height=0.4\textwidth,
    scale only axis,
    ymin=0, xmin=0, xmax=10500,
    xticklabel style={/pgf/number format/fixed},
    yticklabel style={/pgf/number format/fixed},
    scaled x ticks=false
    ]
\addplot[only marks] table[col sep=comma]{binroundlog.csv};
\addplot[domain=0:9999]{-22.6771 + 0.0583*x};
\end{axis}
\end{tikzpicture}


\subsection{Iterating the Logistic Map}
\label{aa:Logistic}
Working precision $m=\mu(50)$ 
and running time (in seconds, bottom) for calculating $k$
reliable iterations of the Logistic Map (Example~\ref{x:Logistic2}):
first with linear fit $10\,000+2.05k$,
second with quadratic fit $22\text{nsec}\cdot k^2-1.2\text{msec}\cdot k$.

\noindent\begin{tikzpicture}
\begin{axis}
[
    xlabel={$k$}, 
    ylabel={},
    xtick={0, 100000, 500000, 1000000},
    width=0.80\textwidth,
    height=0.4\textwidth,
    scale only axis,
    yticklabel style={/pgf/number format/fixed},
    scaled y ticks=false,
    scaled x ticks=false
    ]
\addplot[only marks] table[col sep=comma]{logistic.csv};
\addplot[domain=250:1000000]{10000+2.05*x};
\end{axis}
\end{tikzpicture}

\medskip\noindent\quad\begin{tikzpicture}
\begin{axis}
[
    xlabel={$k$}, 
    ylabel={},
    xtick={0, 100000, 500000, 1000000},
    width=0.80\textwidth,
    height=0.4\textwidth,
    scale only axis,
    yticklabel style={/pgf/number format/fixed},
    scaled y ticks=false,
    scaled x ticks=false
    ]
\addplot[only marks] table[col sep=comma, x index=0,y index=2]{logistic.csv};
\addplot[domain=250:1000000]{2.2E-8*x^2-0.0017*x};
\end{axis}
\end{tikzpicture}

\subsection{Exponentially Steep Function $1/\ln(e/x)$}
\label{aa:Steep}

Working precision $m=\mu(50)$ 
and running time (in seconds, bottom) for approximating $1/\ln(e/x)$
according to Appendix~\ref{a:Steep}:
first with linear fit, 
second with quadratic fit.

\noindent\begin{tikzpicture}
\begin{axis}
[
    xlabel={$\ln(1/x)$}, 
    ylabel={},
    xtick={000000, 2000000, 5000000, 10000000},
    width=0.80\textwidth,
    height=0.4\textwidth,
    scale only axis,
    yticklabel style={/pgf/number format/fixed},
    scaled y ticks=false,
    scaled x ticks=false
    ]
\addplot[only marks] table[col sep=comma, x index=0,y index=2]{steep.csv};
\addplot[domain=250:10000000]{56770+2.260*x};
\end{axis}
\end{tikzpicture}

\medskip\noindent\quad\begin{tikzpicture}
\begin{axis}
[
    xlabel={$\ln(1/x)$}, 
    ylabel={},
    xtick={000000, 2000000, 5000000, 10000000},
    width=0.80\textwidth,
    height=0.4\textwidth,
    scale only axis,
    yticklabel style={/pgf/number format/fixed},
    scaled y ticks=false,
    scaled x ticks=false
    ]
\addplot[only marks] table[col sep=comma, x index=0,y index=1]{steep.csv};
\addplot[domain=250:10000000]{12+5.3E-6*x+8.7E-11*x^2};
\end{axis}
\end{tikzpicture}


\subsection{Runtime Measurements Raw Data}
\label{ss:Data}

\begin{figure}[htb]
\begin{minipage}[t]{0.4\textwidth}
\begin{tabular}{r@{\;\;}|r}%
    \bfseries $x$ ~ & \bfseries time/msec~ 
    \begin{filecontents*}{unaryround.csv}
0.0001,6.155075
-0.0001,5.182057
0.99999,4.562544
-0.99999,4.392102
2,4.188037
-2,4.181838
2.5,3.942352
-2.5,3.956978
2.0001,3.787837
-2.0001,3.901077
2.99999,3.757478
-2.99999,3.816811
1000,4.485281
-1000,4.614206
16777216,6323.040731
-16777216,6929.443467
100,2.052199
200,1.911965
500,2.001556
1000,2.211639
2000,2.659498
5000,3.62145
10000,5.590385
20000,12.472692
50000,22.021783
100000,40.797577
200000,74.125099
500000,197.428489
1000000,394.071932
2000000,863.55783
5000000,2768.039595
10000000,5922.540295
\end{filecontents*}
\csvreader[no head]{unaryround.csv}{}{\\\csvcoli & \csvcolii}
\end{tabular}
\end{minipage}\hfill\begin{minipage}[t]{0.4\textwidth}
\begin{tabular}{r@{\;}|r}%
    \small\bfseries $x$ ~ & \small\bfseries time/msec%
    \begin{filecontents*}{binround.csv}
0.0001,4.076667
-0.0001,5.016125
0.99999,4.76466
-0.99999,4.440129
2.5,4.039695
-2.5,3.850079
2.0001,3.937908
-2.0001,3.865386
2.99999,3.816177
-2.99999,3.639165
1000,3.754206
-1000,3.986979
16777216,3.713004
-16777216,3.735682
268435456,3.724717
-268435456,3.72116
1073741824,3.78413
-1073741824,3.773491
4294967296,3.729712
-4294967296,3.816748
1E+100,5.662915
1E+200,9.093034
1E+300,12.064099
1E+400,13.959395
1E+500,15.14125
1E+600,12.591542
1E+700,16.02842
1E+800,17.709067
1E+900,21.902216
1E+1000,24.110668
1E+2000,58.825866
1E+3000,112.257813
1E+4000,156.610285
1E+5000,221.570062
1E+6000,301.822754
1E+7000,322.746488
1E+8000,431.378636
1E+9000,581.370834
1E+10000,622.362537
\end{filecontents*}
\csvreader[no head]{binround.csv}{}{%
    \\[-0.5ex]\tiny\csvcoli & \tiny\csvcolii}
\end{tabular}
\end{minipage}
\caption{\label{f:RoundTables}\normalsize Running times
of the `unary' (left) 
and the `binary' (right) multivalued integer rounding algorithms
from Examples~\ref{x:URound} and \ref{x:BRound}}
\end{figure}

\normalsize 

\begin{figure}[htb]
\begin{tabular}{r@{\;\;}|r@{\;\;}r}%
    \bfseries $k$ ~ & \bfseries -work.prec & \bfseries time/sec~ 
    \begin{filecontents*}{logistic.csv}
100,	242,	0.00340499
200,	375,	0.00216399
500,	1008,	0.00471899
1000,	2243,	0.00704499
2000,	4657,	0.0212239
5000,	14807,	0.0962719
10000,	57301,	1.4745333
20000,	57301,	2.20068
50000,	112194,	20.4299
100000,	219405,	113.383
200000,	428801,	580.73
500000,	1047294,	4683.64
1000000,	2045771,	20491
\end{filecontents*}
\csvreader[no head]{logistic.csv}{}{\\\csvcoli & \csvcolii & \csvcoliii }
\end{tabular}
\caption{\label{f:LogisticTable}\normalsize Worst-case
working precision and running time when calculating $k$ 
reliable iterations of the Logistic Map (Example~\ref{x:Logistic2})}
\end{figure}

\end{document}